\documentclass[10pt]{amsart}
\usepackage[latin9]{inputenc}
\usepackage{color}
\usepackage{amsbsy}
\usepackage{amstext}
\usepackage{amsthm}
\usepackage{amssymb}
\PassOptionsToPackage{normalem}{ulem}
\usepackage{ulem}

\makeatletter
\numberwithin{equation}{section}
\numberwithin{figure}{section}

\usepackage{tikz-cd}
\usepackage{fancyhdr}
\usepackage{latexsym}
\usepackage{amscd}\usepackage{amsfonts}\usepackage{pstricks}
\usepackage{young}
\usepackage{enumitem}
\usepackage{amsthm}
\usepackage{mathrsfs}

\usepackage{ifthen}
\usepackage{longtable}
\usepackage{todonotes}
\usepackage{marginnote}

\renewcommand{\subsection}[1]{\vspace{3mm}\refstepcounter{subsection}\noindent{\bf \thesubsection. #1.} }

\renewcommand{\subsubsection}[1]{\vspace{3mm}\refstepcounter{subsubsection}\noindent{\bf \thesubsubsection. #1.} }

\numberwithin{equation}{section}

\renewcommand{\labelenumi}{{{\rm(\roman{enumi})}}}

\newtheorem{theorem}{Theorem}\newtheorem{lemma}[theorem]{Lemma}\newtheorem{corollary}[theorem]{Corollary}\newtheorem{proposition}[theorem]{Proposition}

\theoremstyle{definition}
\newtheorem{definition}[theorem]{Definition}\newtheorem{remark}[theorem]{Remark}

\def\CC{\mathbb C}
\def\PP{\mathbb P}

\def\ord{\operatorname{ord}}

\def\min{\mathop{\mathrm{min}}}
\def\CC{\mathbb C}

\def\NN{\mathbb N}
\def\ZZ{\mathbb Z}
\def\PP{\mathbb P}
\def\K{K}

\def\cal{\mathcal }
 \def\ord{\text{ord}}
 
\def\p{\mathbf p}

\def\gen{\mathfrak g}

\let\a\alpha
\let\b\beta

\let\d\delta

\let\g\gamma

\makeatother

\begin{document}
\title[On Pisot's $d$-th root conjecture for function fields  and related GCD estimates]{On Pisot's $d$-th root conjecture for function fields \\
 and related GCD estimates}
\author{Ji Guo}
\address{Institute of Mathematics\\
 Academia Sinica\\
 6F, Astronomy-Mathematics Building\\
 No. 1, Sec. 4, Roosevelt Road \\
 Taipei 10617\\
 Taiwan}
\email{jiguo@math.sinica.edu.tw}
\author{Chia-Liang Sun}
\address{Institute of Mathematics\\
 Academia Sinica\\
 6F, Astronomy-Mathematics Building\\
 No. 1, Sec. 4, Roosevelt Road \\
 Taipei 10617\\
 Taiwan}
\email{csun@math.sinica.edu.tw}
\author{Julie Tzu-Yueh Wang}
\address{Institute of Mathematics\\
 Academia Sinica\\
 6F, Astronomy-Mathematics Building\\
 No. 1, Sec. 4, Roosevelt Road \\
 Taipei 10617\\
 Taiwan}
\email{jwang@math.sinica.edu.tw}

\thanks{2020\ {\it Mathematics Subject Classification}: Primary 11D61; Secondary 14H05 and 11B37}
\thanks{The second named author was supported in part by Taiwan's MoST grant 107-2115-M-001-013-MY2, and the third named author was supported in part by Taiwan's MoST grant 108-2115-M-001-001-MY2.}

\begin{abstract}
We propose a function-field analog of Pisot's $d$-th root conjecture
on linear recurrences, and prove it under some ``non-triviality''
assumption. Besides a recent result of Pasten-Wang on B{\"u}chi's
$d$-th power problem, our main tool, which is also developed in this
paper, is a function-field analog of an GCD estimate in a recent work
of Levin and Levin-Wang. As an easy corollary of such GCD estimate,
we also obtain an asymptotic result. 
\end{abstract}

\maketitle
\baselineskip=16truept


\section{Introduction }

For a field $k$, we denote by $\overline{k}$ its algebraic closure.
Let $R(X)=\sum_{n=0}^{\infty}b(n)X^{n}$ represent a rational function
in $\mathbb{Q}(X)$ and suppose that $b(n)$ is a $d$-th power in
$\mathbb{Q}$ for all large $n\in\mathbb{N}$. Pisot's $d$-th root
conjecture states that one can choose a $d$-th root $a(n)$ of $b(n)$
such that $\tilde{R}(X):=\sum a(n)X^{n}$ is again a rational function
in $\overline{\mathbb{Q}}(X)$. The sequence $\{b(n)\}$ coming from
the rational function $R(X)$ is a linear recurrence sequence, which
can be written as an \emph{exponential polynomial}, which we now define.
An \emph{exponential polynomial over a field $k$} is a sequence $b:\mathbb{N}\rightarrow k$
of the form 
\begin{equation}
n\mapsto\sum_{i\in I}B_{i}(n)\beta_{i}^{n},\label{eq:exp-poly}
\end{equation}
where $I$ is a (finite) set of indices, each $\beta_{i}\in k^{*}$
is nonzero and each $B_{i}\in k[T]$ is a single-variate polynomial.
When it can be arranged so that each $B_{i}$ is constant, we say
that $b$ is simple. For any exponential polynomial $b:\mathbb{N}\rightarrow k$
and any natural number $d$, it is clear that $b^{d}:\mathbb{N}\rightarrow k$,
defined by $n\mapsto b(n)^{d}$, is still an exponential polynomial.
The following result of Zannier \cite{Zannier2000} essentially proves
its converse, which is a generalization of Pisot's $d$-th root conjecture
stated earlier. 
We also refer to \cite{Zannier2000} for a survey on related works.

\begin{theorem}[\cite{Zannier2000}] \label{thm:Pisot's_Conj}Let $b$
be an exponential polynomial over a number field $k$, and $d\ge2$
be an integer. Suppose that $b(n)$ is the $d$-th power of some element
in $k$ for all but finitely many $n$. Then there exists an exponential
polynomial $a$ over $\overline{k}$ such that $a(n)^{d}=b(n)$ for
all $n$. \end{theorem}

The main purpose of this paper is to investigate a function-field
analog of Theorem \ref{thm:Pisot's_Conj}.

Let $C$ be a smooth projective algebraic curve of genus $\mathfrak{g}$
defined over an algebraically closed field ${\bf k}$ of characteristic
zero. Let $K:=\mathbf{k}(C)$ be its function field. We will always
denote by $\mathfrak{p}$ a point in $C(\mathbf{k})$, by $S$ a finite
subset of $C(\mathbf{k})$. Since $K$ contains the algebraically-closed
field $\mathbf{k}$, note that for any pair of exponential polynomials
$a:\mathbb{N}\rightarrow K$ and $c:\mathbb{N}\rightarrow\mathbf{k}$
we have that $ca^{d}:\mathbb{N}\rightarrow K$, defined by $n\mapsto c(n)a(n){}^{d}$,
is still an exponential polynomial whose $n$-th term is the $d$-th
power of some element in $K$ for all $n\in\mathbb{N}$. A plausible
statement obtained from Theorem \ref{thm:Pisot's_Conj} by replacing
the number field $k$ by our function field $K$ must therefore have
its conclusion modified to the existence of exponential polynomials
$a$, $c$, respectively over $\overline{K}$ and over $\mathbf{k}$,
such that $c(n)a(n)^{d}=b(n)$ for all $n$. Our result in this direction
is as follows. 
 \begin{theorem} \label{dPisot} Let $b(n)=\sum_{i=1}^{\ell}B_{i}(n)\beta_{i}^{n}$
be an exponential polynomial over $K$, i.e. $B_{i}\in K[T]$ and
$\beta_{i}\in K^{*}$. Let $\Gamma$ be the multiplicative subgroup
of $K^{*}$ generated by $\beta_{1},\hdots,\beta_{\ell}$. Assume
that $\Gamma\cap{\bf k}^{*}=\{1\}$. If $b(n)$ is a $d$-th power
in $K$ for infinitely many $n\in\NN$, then there exists an exponential
polynomial $a(m)=\sum_{i=1}^{r}A_{i}(m)\alpha_{i}^{m}$, $A_{i}\in\overline{K}[T]$, $\a_{i}\in\overline{K}^{*}$,
and a polynomial $R\in{\bf k}[T]$ such that $b(m)=R(m)a(m)^{d}$
for all $m\in\mathbb{N}$. \end{theorem}

\begin{remark} The assumption that $b(n)$ is a $d$-th power in
$K$ for infinitely any $n\in\NN$ is weaker than the one in Theorem
\ref{thm:Pisot's_Conj}. 
We refer to \cite{CZ1998} for a result over number field under both
this weaker assumption and the existence of a dominant $\beta_{i}$,
i.e. a unique $\beta_{i}$ of maximal or minimal absolute value. \end{remark}

\begin{remark} \label{rem: tor}

In the notation of Theorem \ref{dPisot}, it is standard to notice
that $b$ consists of the $q$ disjoint subsequences $b_{j}$, defined
by $n\mapsto b(j+qn)$, where $j\in\{0,\ldots,q-1\}$ and $q$ is
the order of the torsion subgroup of $\Gamma$; moreover, each $b_{j}$
is an exponential polynomial whose associated $\Gamma_{j}$ is torsion-free.
With this observation, we may generalize Theorem \ref{dPisot} so
that the assumption $\Gamma\cap{\bf k}^{*}=\{1\}$ is relaxed to that
$\Gamma\cap{\bf k}^{*}$ is finite and the conclusion only holds for
some $b_{j}$ rather than $b$. 

\end{remark}

\begin{remark} \label{rem: rk1}

In the case where $b$ is simple, i.e., each $B_{i}$ is constant,
we can relax the hypothesis on $\Gamma$ in Theorem \ref{dPisot}
so that the case where $\Gamma\cap{\bf k}^{*}$ is infinitely cyclic
generated by $\gamma$ can be also treated. In this new case, modifying
slightly our proof of Theorem \ref{dPisot}, we can conclude 
that $b(m)=c(m)a(m)^{d}$ for all $m\in\mathbb{N}$, where $c$ is
a simple exponential polynomial over $\mathbf{k}$ given by $m\mapsto\sum_{i=1}^{r}c_{i}\gamma^{e_{i}m}$
for some $c_{i}\in{\bf k}$ and $e_{i}\in\mathbb{Z}$. It seems difficult
to further relax the hypothesis on $\Gamma$.

\end{remark}



Our proof of Theorem \ref{dPisot} contains two major ingredients,
both rely on the special features of function fields of characteristic
zero. One of the ingredients is the result (restated as Theorem \ref{buchi_func_field}
in Section \ref{section4}) by Pasten and Wang \cite{PW2015} motivated
by B{\"u}chi's $d$-th power problem, which has a similar flavor as
Pisot's $d$-th root conjecture but arising from different purposes.
While working on an undecidability problem related to Hilbert's tenth
problem in the 1970s, B{\"u}chi formulated a related arithmetic problem,
which can be stated in more generality as follows: Let $k$ be a number
field. Does there exist an integer $M$ such that the only monic polynomials
$G\in k[T]$ of degree $d$ satisfying that $G(1),\hdots,G(M)$ are
$d$-th powers in $k$, are precisely those of the form $G(X)=(X+c)^{d}$
for some $c\in k$. This problem remains unsolved, while its analogs
have been investigated intensively in the recent years. In particular,
the analog over function fields of characteristic zero was solved
completely, even with an explicit bound on $M$; see \cite{Buchi2013}
and \cite{Pa}. We refer to \cite{PW2015} for a survey of relevant
works. The other ingredient, which is also developed in this paper,
is the function-field analog of the recent work of Levin \cite{Levin:GCD}
for number fields and Levin-Wang \cite{levin2019greatest} for meromorphic
functions on GCD estimates of two multivariable polynomials over function
fields evaluated at  arguments which are $S$-units. We will use the estimates
through the following result, which is of independent interest.

\begin{theorem} \label{thm:y^e=00003D00003D00003D00003D00003DFg^l}
Let $d\ge2$ be an integer and $F\in K[x_{1},\dots,x_{n}]$. Assume
that  $F$ can not be expressed   as $a\mathbf{x}^{\mathbf{i}}G^{d}$ for any $a\in K^{*}$,
any monomial $\mathbf{x}^{\mathbf{i}}\in K[x_{1},\dots,x_{n}]$, and
any $G\in K[x_{1},\dots,x_{n}]$. Then we have the following conclusion:
For any $u_{1},\hdots,u_{n}\in\mathcal{O}_{S}^{*}$, there exists
positive integer $m$ and rationals $c_{1}$ and $c_{2}$ 
all depending only on $\left(d,n,\deg F,\max_{1\le j\le n}h(u_{j})\right)$,
such that if $F(u_{1}^{\ell},\ldots,u_{n}^{\ell})$ is a $d$-th power
in $K$ with some $\ell\ge c_{1}\tilde{h}(F)+c_{2}\max\{1,2\gen-2+|S|\}$,
then $u_{1}^{m_{1}}\cdots u_{n}^{m_{n}}\in{\bf k}$ for some $(m_{1},\hdots, m_{n})\in\mathbb{Z}^{n}\setminus\{(0,\ldots,0)\}$
with $\sum|m_{i}|\le2m$.
\end{theorem} Here $\tilde{h}(F)$ is the
relevant height of $F$ to be defined in the next section.

{\begin{remark} We cannot drop $a$ in the assumption
of Theorem \ref{thm:y^e=00003D00003D00003D00003D00003DFg^l}. For
example, if $a=u_{1}\in\mathcal{O}_{S}^{*}$ and $F(x_{1}):=ax_{1}$,
we always have that $F(u_{1}^{d\ell-1})$ is a $d$-th power in $K$
for all $\ell\in\mathbb{N}$.\end{remark}}

The assumption in Theorem \ref{dPisot} that $\Gamma\cap\mathbf{k}^{*}$
is trivial implies that every minimal set of generators of $\Gamma$
is multiplicatively independent modulo ${\bf k}$. This suggests how
Theorem \ref{thm:y^e=00003D00003D00003D00003D00003DFg^l} plays a
role in our proof of Theorem \ref{dPisot}.

We briefly describe the core idea connecting GCD estimates and our proof
of Theorem \ref{thm:y^e=00003D00003D00003D00003D00003DFg^l}, as introduced
by Corvaja-Zannier \cite{CZ2008}.    After
some reduction, we only need to treat the case where $F$ is $d$-th
power free. Given   a tuple $(u_{1},\hdots,u_{n},y)\in(\mathcal{O}_{S}^{*})^{n}\times K$
satisfies that   $y^{d}=F(u_{1},\ldots,u_{n})$, we
will construct a polynomial $G\in K[x_{1},\dots,x_{n}]$ with controllable
height, depending on  $F$  and the $\frac{u_{i}'}{u_{i}}$, such that $(y^{d})'=G(u_{1},\ldots,u_{n})$, where
$'$ denotes a global derivation on $K$. For example, if
$F:=x_{1}^{2}+\cdots+x_{n}^{2}$, then our construction will yield
$G:=2\frac{u_{1}'}{u_{1}}x_{1}^{2}+\hdots+2\frac{u_{n}'}{u_{n}}x_{n}^{2}.$
As $d\ge2$, the number of common zeros of $y^{d}$ and 
$(y^{d})'$ is essentially larger than the number of zeros of $y^{d-1}$.
 On the other hand, we expect
the number of common zeros of $F(u_{1},\ldots,u_{n})$ and $G(u_{1},\ldots,u_{n})$
to be essentially smaller than the number of zeros of $y^{d-1}$ unless something special
happens. To formalize this idea, we prove the following result on
GCD estimates, where all notation involved are defined in Section
\ref{sec:Preliminaries}. 

\begin{theorem} \label{movinggcdunit}
Let $S\subset C$ be a finite set of points. Let $F,\,G\in K[x_{1},\dots,x_{n}]$
be a coprime pair of nonconstant polynomials. For any $\epsilon>0$,
there exist an integer $m$, positive reals $c_{i}$, $0\le i\le4$,
all depending only on $\epsilon$, such that for all $n$-tuple $(g_{1},\hdots,g_{n})\in({\cal O}_{S}^{*})^{n}$
with 
\[
\max_{1\le i\le n}h(g_{i})\ge c_{1}(\tilde{h}(F)+\tilde{h}(G))+c_{2}\max\{0,2\gen-2+|S|\},
\]
we have that either 
\begin{align}
h(g_{1}^{m_{1}}\cdots g_{n}^{m_{n}})\le c_{3}(\tilde{h}(F)+\tilde{h}(G))+c_{4}\max\{0,2\gen-2+|S|\}\label{multiheight11}
\end{align}
holds for some integers $m_{1},\hdots,m_{n}$, not all zeros with $\sum|m_{i}|\le2m$,
or the following two statements hold.

\global\long\def\labelenumi{(\theenumi)}%
\global\long\def\theenumi{\alph{enumi}}%

\begin{enumerate}[label=(\alph*)] 
\item[{\rm(a)}] \label{enu:;Nsgcd<} $N_{S,{\rm gcd}}(F(g_{1},\hdots,g_{n}),G(g_{1},\hdots,g_{n}))\le\epsilon\max_{1\le i\le n}h(g_{i})$; 
\item[{\rm(b)}] \label{enu: hgcd<} $h_{{\rm gcd}}(F(g_{1},\hdots,g_{n}),G(g_{1},\hdots,g_{n}))\le\epsilon\max_{1\le i\le n}h(g_{i})$
if we further assume that not both of $F$ and $G$ vanish at $(0,\hdots,0)$. 
\end{enumerate}
\end{theorem}

\begin{remark} In Theorem \ref{movinggcdunit}, all the quantities
claimed to exist can be given effectively. Moreover, the explicit
bounds on heights are important in our application. As said earlier,
if we are given $F=x_{1}^{2}+\cdots+x_{n}^{2}$ in Theorem \ref{thm:y^e=00003D00003D00003D00003D00003DFg^l},
then in the main step of the proof, we construct $G_{\mathbf{u}}:=2\frac{u_{1}'}{u_{1}}x_{1}^{2}+\hdots+2\frac{u_{n}'}{u_{n}}x_{n}^{2}$
for each tuple $\mathbf{u}:=(u_{1},\hdots,u_{n})\in({\cal O}_{S}^{*})^{n}$
and apply Theorem \ref{movinggcdunit} to estimate the GCD of $F(\mathbf{u})$
and $G_{\mathbf{u}}(\mathbf{u})$. The main point which makes the
proof of Theorem \ref{thm:y^e=00003D00003D00003D00003D00003DFg^l}
works is that $\tilde{h}(G_{{\bf u}})$ can be explicitly bounded
independent of these $\mathbf{u}$. (See Proposition \ref{heightDu}.)
\end{remark}


It is more desirable to obtain GCD estimates, such as Statement (a)
and (b) in Theorem \ref{movinggcdunit}, under the assumption
that $g_{1},\hdots,g_{n}$ are multiplicatively independent modulo
${\bf k}$. As a result in this direction, we can actually replace
the right hand side of \eqref{multiheight11} by $0$ in the case
where $n=2$ and the coefficients of $F$ and $G$ are in ${\bf k}$.
We include a complete statement below. Although this result can be
deduced from \cite[Corollary 2.3]{CZ2008}, we will derive it from
our proof of Theorem \ref{movinggcdunit}. \begin{theorem} \label{n=00003D00003D00003D00003D00003D00003D2gcdunit}
Let $F,\,G\in{\bf k}[x_{1},x_{2}]$ be nonconstant coprime polynomials.
For any $\epsilon>0$, there exist an integer $m$, constant $c$,
both depending only on $\epsilon$, such that for all pairs $(g_{1},g_{2})\in({\cal O}_{S}^{*})^{2}$
with $\max\{h(g_{1}),h(g_{2})\}\ge c\,\max\{1,2\gen-2+|S|\}$, either we have that
$g_{1}^{m_{1}}g_{2}^{m_{2}}\in{\bf k}$ holds for some integers $m_{1},m_{2}$,
not all zeros with $|m_{1}|+|m_{2}|\le2m$, or the following two statements
hold 

\global\long\def\labelenumi{(\theenumi)}%
\global\long\def\theenumi{\alph{enumi}}%

\begin{enumerate} 
\item[{\rm(a)}] $N_{S,{\rm gcd}}(F(g_{1},g_{2}),G(g_{1},g_{2}))\le\epsilon\max\{h(g_{1}),h(g_{2})\}$; 
\item[{\rm(b)}] $h_{{\rm gcd}}(F(g_{1},g_{2}),G(g_{1},g_{2}))\le\epsilon\max\{h(g_{1}),h(g_{2})\}$,
if we further assume that not both of $F$ and $G$ vanish at $(0,0)$. 
\end{enumerate}
\end{theorem}

As another result in the same direction, we obtain easily from Theorem
\ref{movinggcdunit} that an \emph{effectively} asymptotic version
of Statement (a) and (b) in Theorem \ref{movinggcdunit}
holds, \emph{merely assuming that }$g_{1},\hdots,g_{n}$ are multiplicatively
independent modulo ${\bf k}$; here the effectivity means that we
have an effective lower bound for $\ell$ in the following statement.
\begin{theorem} \label{movinggcdpower} Let $F,\,G\in\K[x_{1},\dots,x_{n}]$
be nonconstant coprime polynomials. Let $g_{1},\hdots,g_{n}\in\K^{*}$,
not all constant. Then for any $\epsilon>0$, there exist an integer
$m$ and constant $c_{1}$ and $c_{2}$ depending only on $\epsilon$,
such that for each positive integer 
\[
\ell>c_{1}(\tilde{h}(F)+\tilde{h}(G))+c_{2}(\gen+n\max_{1\le i\le n}\{h(g_{i})\}),
\]
either we have $g_{1}^{m_{1}}\cdots g_{n}^{m_{n}}\in{\bf k}$ for
some integers $m_{1},\hdots,m_{n}$, not all zeros with $\sum|m_{i}|\le2m$,
or the following two statements hold.

\global\long\def\labelenumi{(\theenumi)}%
\global\long\def\theenumi{\alph{enumi}}%

\begin{enumerate}[label=(\alph*)]
\item[{\rm(a)}] $N_{S,{\rm gcd}}(F(g_{1}^{\ell},\hdots,g_{n}^{\ell}),G(g_{1}^{\ell},\hdots,g_{n}^{\ell}))\le\epsilon\max_{1\le i\le n}h(g_{i}^{\ell})$; 
\item[{\rm(b)}] $h_{{\rm gcd}}(F(g_{1}^{\ell},\hdots,g_{n}^{\ell}),G(g_{1}^{\ell},\hdots,g_{n}^{\ell}))\le\epsilon\max_{1\le i\le n}h(g_{i}^{\ell})$,
if we further assume that not both of $F$ and $G$ vanish at $(0,\hdots,0)$. 
\end{enumerate}
\end{theorem}

\begin{remark} When $F,\,G\in\mathbb{C}[x_{1},\dots,x_{n}]$ be a
coprime pair of nonconstant polynomials and $g_{1},\ldots,g_{n}\in\CC[z]$
are multiplicatively independent modulo $\CC$, then the results in
\cite{levin2019greatest} also imply Statement (a) and (b)  in Theorem \ref{movinggcdpower}. Our statement here is stronger
since we have formulated effective bounds on $\ell$ and the $m_{i}$
such that $g_{1}^{m_{1}}\cdots g_{n}^{m_{n}}\in{\bf k}$. When $n>2$,
the only other previous result in this direction appears to be a result
of Ostafe \cite[Th.~1.3]{Ostafe}, which considers special polynomials
such as $F=x_{1}\cdots x_{r}-1,G=x_{r+1}\cdots x_{n}-1$, but proves
a  uniform bound   in place of Statement (a) and (b)  independent of $\ell$. In the  case where $n=2$,
previous results include the original theorem of Ailon-Rudnick \cite{AR}
in this setting, i.e. $F=x_{1}-1$, $G=x_{2}-1$, and extensions of
Ostafe \cite{Ostafe} and Pakovich and Shparlinski \cite{PS} (all
with uniform bounds). It is noted in \cite{Ostafe} that it appears
to be difficult to extend the techniques used there to obtain results
for general $F$ and $G$. \end{remark}

We collect the background materials in Section \ref{sec:Preliminaries}.
We will prove some main lemmas in Section \ref{mainlemmas}. The proofs
of Theorem \ref{dPisot} and Theorem \ref{thm:y^e=00003D00003D00003D00003D00003DFg^l}
are given in Section \ref{sectionPisot} and Section \ref{section4}
respectively. Finally, we establish the gcd theorems in Section \ref{gcd}.

\section{Preliminaries}

\label{sec:Preliminaries}

Recall that $K$ is the function field of the smooth projective curve
$C$ of genus $\mathfrak{g}$ defined over the algebraically closed
field ${\bf k}$ of characteristic zero. At each point $\p\in C(\mathbf{k})$,
we may choose a uniformizer $t_{\p}$ to define a normalized order
function $v_{\p}:=\ord_{\p}:\K\to\ZZ\cup\{+\infty\}$. 
Let $S\subset C(\mathbf{k})$ be a finite subset. We denote the ring
of $S$-integers in $K$ and the group of $S$-units in $K$ respectively
by 
\[
{\cal O}_{S}:=\{f\in\K\,|\,v_{\p}(f)\ge0\text{ for all }\p\notin S\},
\]
and 
\[
{\cal O}_{S}^{*}:=\{f\in\K\,|\,v_{\p}(f)=0\text{ for all }\p\notin S\}.
\]

For simplicity of notation, for $f\in\K^{*}$ and $\mathbf{p}\in C(\mathbf{k})$
we let 
\[
v_{\p}^{0}(f):=\max\{0,v_{\p}(f)\},\qquad\bar{v}_{\p}^{0}(f):=\min\{1,v_{\p}^{0}(f)\},
\]
i.e. its order of zero at $\p$ and its truncated value; 
\[
v_{\p}^{\infty}(f):=-\min\{0,v_{\p}(f)\},\qquad\bar{v}_{\p}^{\infty}(f):=\min\{1,v_{\p}^{\infty}(f)\},
\]
i.e. its order of pole at $\p$ and its truncated value. The height
of $f$ is defined by 
\[
h(f):=\sum_{\p\in C }-v_{\p}^{\infty}(f),
\]
which counts the number of poles of $f$ with multiplicities. For
any ${\bf f}:=[f_{0}:\cdots:f_{m}]\in\PP^{m}(K)$ with $m\ge1$ and
$f_{0},...,f_{m}\in\K$, we define $v_{\p}(\mathbf{f}):=\min\{v_{\p}(f_{0}),...,v_{\p}(f_{m})\}$
and 
\[
h({\bf f})=h(f_{0},...,f_{m}):=\sum_{\p\in C}-v_{\p}(\mathbf{f}).
\]
For a finite subset $S$ of $C$ and $f\in K^{*}$, we let 
\[
\overline{N}_{S}({f})={\displaystyle \sum_{\mathbf{p}\in C\setminus S}\bar{v}_{\mathbf{p}}^{0}(f)}.
\]
be the cardinality of the set of zeros of $f$ outside $S$; and 
\[
N_{S}({f})=\sum_{\p\notin S}v_{\p}^{0}(f)
\]
be the number of the zero, counting multiplicities, of $f$ outside
of $S$. For any $f,g\in\K,$ we let 
\begin{align*}
N_{S,{\rm gcd}}(f,g):=\sum_{\p\in C\setminus S}\min\{v_{\p}^{0}(f),v_{\p}^{0}(g)\}
\end{align*}
and 
\begin{align*}
h_{{\rm gcd}}(f,g):=\sum_{\p\in C}\min\{v_{\p}^{0}(f),v_{\p}^{0}(g)\}.
\end{align*}

Let $\mathbf{x}:=(x_{1},\ldots,x_{n})$ be a tuple of $n$ variables,
and $F=\sum_{{\bf i}\in I_{F}}a_{{\bf i}}{\bf x}^{{\bf i}}\in K[{\bf x}]$
be a nonzero polynomial, where $I_{F}$ is the (nonempty finite) set
of those indices ${\bf i}=(i_{1},\hdots,i_{n})$ with $a_{{\bf i}}\ne0$;
here, each $i_{j}$ is a nonnegative integer, and we put ${\bf x}^{{\bf i}}:=x_{1}^{i_{1}}\cdots x_{n}^{i_{n}}$.
We define the height $h(F)$ and the relevant height $\tilde{h}(F)$
as follows. Put 
\begin{align*}
v_{\p}(F):=\min_{{\bf i}\in I_{F}}\{v_{\p}(a_{{\bf i}})\}\qquad\text{for }\p\in C.
\end{align*}
and define 
\begin{align*}
h(F):=\sum_{\p\in C}-v_{\p}(F),
\end{align*}
\begin{align*}
\tilde{h}(F):=\sum_{\p\in C}-\min\{0,v_{\p}(F)\}.
\end{align*}
Notice that Gauss's lemma can be stated as 
\begin{align*}
v_{\p}(FG)=v_{\p}(F)+v_{\p}(G), 
\end{align*}
where $F$ and $G$ are in $K[x_{1},\hdots,x_{n}]$ and $\p\in C$.
Consequently, we have that 
\begin{align}
h(FG)=h(F)+h(G).\label{Gaussht}
\end{align}
Although relevant height $\tilde{h}(F)$ is not projectively invariant,
it suits better when comparing with the height of an individual coefficient
of $F$. Indeed, we have from the definitions that 
\begin{align}
h(a_{{\bf i}})\le\tilde{h}(F)\quad\text{and}\quad\tilde{h}(a_{{\bf i}}^{-1}F)=h(a_{{\bf i}}^{-1}F)=h(F)\le\tilde{h}(F),\label{compareheight}
\end{align}
where $a_{{\bf i}}$ is any non-zero coefficient $F$.

We now recall the definitions of global and local derivations on $\K$.
Let $t\in\K\setminus {\bf k}$, which will be fixed later. The mapping
${\displaystyle {g\to\frac{dg}{dt}}}$ on ${\bf k}(t)$, formal differentiation
on ${\bf k}(t)$ with respect to $t$, extends uniquely to a global
derivation on $\K$ as $\K$ is a finite separable extension of ${\bf k}(t)$.
Furthermore, since an element in $\K$ can be written as a Laurent
series in $t_{\p}$, the local derivative of $\eta\in\K$
with respect to $t_{\p}$, denoted by ${\displaystyle {d_{\p}\eta:=\frac{d\eta}{dt_{\p}}}}$,
is given by the formal differentiation on ${\bf k}((t_{\p}))$ with
respect to $t_{\p}$. Consequently, 
\begin{align}
\frac{d\eta}{dt}=d_{\p}\eta\cdot(d_{\p}t)^{-1}.\label{chain rule}
\end{align}




The following results are consequences of the Riemann-Roch Theorem.
We refer to \cite[Corollary 7]{Buchi2013} for a proof. 
\begin{proposition}\label{functiont}
For each point $\p_{\infty}\in C$, we can find some $t\in K\setminus{\bf k}$ satisfying the following conditions:

\global\long\def\labelenumi{(\theenumi)}%
\global\long\def\theenumi{\alph{enumi}}%

\begin{enumerate}[label=(\alph*)]
\item[{\rm(a)}]  $t$ has exactly one pole at $\p_{\infty}$; 
\item[{\rm(b)}]  $h(t)\le\gen+1$; 
\item[{\rm(c)}]  \label{prop12_3} ${\displaystyle {\sum_{\p\in C}v_{\p}^{0}(d_{\p}t)\le3\gen}}$. 
\end{enumerate}
\end{proposition}

We will use the following result of Brownawell-Masser \cite{BM}.
\begin{theorem}\label{BrMa} Let the characteristic of $\K$ be zero.
If $f_{1},\hdots,f_{n}\in\mathcal{O}_{S}^{*}$ and $f_{1}+\cdots+f_{n}=1$,
then either some proper subsum of $f_{1}+\cdots+f_{n}$ vanishes or
\[
\max_{1\le i\le n}h(f_{i})\le\frac{n(n-1)}{2}\max\{0,2\gen-2+|S|\}.
\]
\end{theorem} The following is an analogue of Green's lemma in Nevanlinna's
theory. \begin{corollary}\label{Green} Let the characteristic of
$\K$ be zero and $\ell$ be an integer. Let $a_{1},\hdots,a_{n},f_{1},\hdots,f_{n}\in K^{*}$,
$n\ge2$. If 
\[
a_{1}f_{1}^{\ell}+\cdots+a_{n}f_{n}^{\ell}=0,
\]
and no proper subsum of $a_{1}f_{1}^{\ell}+\cdots+a_{n}f_{n}^{\ell}$
vanishes. Then $\frac{f_{i}}{f_{j}}\in{\bf k}$, for all $1\le i,j\le n$,
if 
\begin{align}\label{lbound}
\ell>(n-1)^{2}(n-2)\max\{1,\gen\}+(n-1)^{4}h(a_{1},\hdots,a_{n}).
\end{align}
\end{corollary} \begin{proof} Let $b_{i}=\frac{a_{i}}{a_{n}}$ and
$g_{i}=\frac{f_{i}}{f_{n}}$. Then 
\begin{align}
b_{1}g_{1}^{\ell}+\cdots+b_{n-1}g_{n-1}^{\ell}=-1,\label{reformulate1}
\end{align}
and no proper subsum of $b_{1}g_{1}^{\ell}+\cdots+b_{n-1}g_{n-1}^{\ell}$
vanishes. Suppose that \eqref{lbound} holds and that at least one of the $g_{i}$, $1\le i\le n-1$,
is not constant. Let $S$ be the set consisting of the zeros and poles
of $b_{i}$ and $g_{i}$, $1\le i\le n-1$. Then 
\[
2\le|S|\le2\sum_{i=1}^{n-1}(h(b_{i})+h(g_{i})),
\]
and all the $b_{i}$ and $g_{i}$ are $S$-units. Applying Theorem
\ref{BrMa} to the equation \eqref{reformulate1}, we have 
\begin{align}
h(b_{i}g_{i}^{\ell})\le\frac{(n-1)(n-2)}{2}(2\gen-2+2\sum_{j=1}^{n-1}(h(b_{j})+h(g_{j}))),\label{hestimate1}
\end{align}
for $1\le i\le n-1$. As 
\[
\ell h(g_{i})\le h(b_{i}g_{i}^{\ell})+h(b_{i}),
\]
for $1\le i\le n-1$, together with \eqref{hestimate1} we have 
\[
\ell\sum_{i=1}^{n-1}h(g_{i})\le\sum_{i=1}^{n-1}h(b_{i})+(n-1)^{2}(n-2)(\gen-1+\sum_{i=1}^{n-1}(h(b_{i})+h(g_{i}))).
\]
Hence, 
\[
(\ell-(n-1)^{2}(n-2))\sum_{i=1}^{n-1}h(g_{i})\le(n-1)^{2}(n-2)(\gen-1)+((n-1)^{2}(n-2)+1)\sum_{i=1}^{n-1}h(b_{i}).
\]
Since one of the $g_{i}$ is not constant, $\ell>(n-1)^{2}(n-2)$ by \eqref{lbound}
and $h(b_{i})=h(a_{i},a_{n})\le h(a_{1},\hdots,a_{n})$, it implies
that 
\[
\ell-(n-1)^{2}(n-2)\le(n-1)^{2}(n-2)(\gen-1)+(n-1)^{4}h(a_{1},\hdots,a_{n}),
\]
contradicting to \eqref{lbound}. 
\end{proof}

\section{Main Lemmas}

\label{mainlemmas} From now on, we will fix a $t$ satisfying the
conditions in Proposition \ref{functiont} and use the notation $\eta':=\frac{d\eta}{dt}$
for $\eta\in K$. We will use the follow estimate. \begin{lemma}\label{lem:NSgcd_lb}
Let $S$ be a finite subset of $C$. Then the following statements
hold.

\global\long\def\labelenumi{(\theenumi)}%
\global\long\def\theenumi{\alph{enumi}}%

\begin{enumerate}[label=(\alph*)]
\item[{\rm(a)}]  $N_{S,\gcd}(\eta,\eta')\ge N_{S}(\eta)-\overline{N}_{S}(\eta)-3\mathfrak{g}$
for any $\eta\in K$.
\item[{\rm(b)}]  $h(1,\frac{\eta_{1}'}{\eta_{1}},\hdots,\frac{\eta_{\ell}'}{\eta_{\ell}})\le|S|+3\gen$,
where $\eta_{i}\in\mathcal{O}_{S}^{*}$ for each $1\le i\le\ell$. 
\end{enumerate}
\end{lemma}

\begin{proof} It is clear from \eqref{chain rule} that 
\begin{align}\label{eq: v_of_diff}
 v_{\p}(\eta')&=v_{\p}(\eta)-1-v(d_{\p}t) &\text{ if }v_{\p}(\eta)\ne0;\cr
v_{\p}(\eta')&\ge-v(d_{\p}t)  &\text{ if }v_{\p}(\eta)=0.
\end{align}
Consequently, 
\begin{align*}
N_{S,\gcd}(\eta,\eta') & =\sum_{\p\notin S}\min\{v_{\p}^{0}(\eta),v_{\p}^{0}(\eta')\}\\
 & =\sum_{v_{\p}(\eta)>0,\,\p\notin S}\min\{v_{\p}(\eta),v_{\p}(\eta)-1-v(d_{\p}t)\}\\
 & \ge\sum_{v_{\p}(\eta)>0,\,\p\notin S}(v_{\p}(\eta)-1-v^{0}(d_{\p}t))\\
 & \ge N_{S}(\eta)-\overline{N}_{S}(\eta)-3\mathfrak{g}
\end{align*}
by Proposition \ref{functiont} (c). Again by \eqref{eq: v_of_diff}
and the assumption that $\eta_{i}\in\mathcal{O}_{S}^{*}$ for each
$1\le i\le\ell$, we have

\begin{align*}
h(1,\frac{\eta_{1}'}{\eta_{1}},\hdots,\frac{\eta_{\ell}'}{\eta_{\ell}}) & =\sum_{\p\in C}-\min_{1\le i\le\ell}\{0,v_{\p}(\eta_{i}')-v_{\p}(\eta_{i})\}\\
 & \le\sum_{\p\in S}-\min\{0,-1-v(d_{\p}t)\}+\sum_{\p\in C\setminus S}-\min\{0,-v(d_{\p}t)\}\\
 & \le|S|+\sum_{\p\in C}v^{0}(d_{\p}t)\\
 & \le|S|+3\gen
\end{align*}
by Proposition \ref{functiont} (c). \end{proof}

For convenience of discussion, we will use the following convention.
Let $\mathbf{i}=(i_{1},\ldots,i_{n})\in\mathbb{Z}^{n}$ and $\mathbf{u}=(u_{1},\ldots,u_{n})\in(K^{*})^{n}$.
We denote by $\mathbf{x}:=(x_{1},\ldots,x_{n})$, $\mathbf{x^{i}}:=x_{1}^{i_{1}}\cdots x_{n}^{i_{n}}$,
$\mathbf{u^{i}}:=u_{1}^{i}\cdots u_{n}^{i_{n}}\in K^{*}$ and $|\mathbf{i}|:=\sum_{j=1}^{n}|i_{j}|$.
For a polynomial $F(\mathbf{x})=\sum_{\mathbf{i}}a_{\mathbf{i}}\mathbf{x}^{\mathbf{i}}\in K[x_{1},\dots,x_{n}]$,
we denote by $I_{F}$ the set of exponents ${\bf i}$ such that $a_{\mathbf{i}}\ne0$
in the expression of $F$, and define 
\begin{align}
D_{\mathbf{u}}(F)(\mathbf{x}):=\sum_{\mathbf{i}\in I_{F}}\frac{(a_{\mathbf{i}}\mathbf{u}^{\mathbf{i}})'}{\mathbf{u}^{\mathbf{i}}}\mathbf{x}^{\mathbf{i}}.\label{DuF}
\end{align}
Clearly, we have 
\begin{align}
F(\mathbf{u})'=D_{\mathbf{u}}(F)(\mathbf{u}),\label{value}
\end{align}
and the following product rule: 
\begin{align}
D_{\mathbf{u}}(FG)=D_{\mathbf{u}}(F)G+FD_{\mathbf{u}}(G)\label{product}
\end{align}
for each $F,G\in K[x_{1},\dots,x_{n}]$.

The following proposition gives an upper bound on height of the coefficients
of $D_{\mathbf{u}}(F)$ when the $u_{i}$'s are $S$-units. This is
a crucial step. \begin{proposition}\label{heightDu} Let $F$ be
a nonconstant polynomial in $K[x_{1},\dots,x_{n}]$ and $\mathbf{u}=(u_{1},\ldots,u_{n})\in(O_{S}^{*})^{n}$.
Then there exist $c_{1},c_{2}$ depending only on $\deg F$
such that 
\[
\tilde{h}(D_{\mathbf{u}}(F))\le c_{1}\tilde{h}(F)+c_{2}\max\{1,2\gen-2+|S|\}.
\]
\end{proposition} \begin{proof} Let $F(x_{1},\ldots,x_{n})=\sum_{\mathbf{i}\in I_{F}}a_{\mathbf{i}}\mathbf{x}^{\mathbf{i}}$.
We then choose $S'$ containing $S$ and all the zeros and poles of
all $a_{\mathbf{i}}$ for $\mathbf{i}\in I_{F}$. Then 
\begin{align}
|S'|\leq|S|+2\sum_{\mathbf{i}\in I_{F}}h(a_{\mathbf{i}})\le|S|+2|I_{F}|\tilde{h}(F)\label{Sestimate}
\end{align}
and $a_{\mathbf{i}}\in O_{S'}^{*}$ for each $\mathbf{i}\in I_{F}$.
As 
\[
D_{\mathbf{u}}(F)(\mathbf{x})=\sum_{\mathbf{i}\in I_{F}}a_{\mathbf{i}}\cdot\frac{(a_{\mathbf{i}}\mathbf{u^{i}})'}{a_{\mathbf{i}}\mathbf{u^{i}}}\mathbf{x}^{\mathbf{i}},
\]
we have that 
\begin{align}\label{hFduF}
\tilde{h}(D_{\mathbf{u}}(F)) & \le h(1,(a_{\mathbf{i}})_{{\mathbf{i}}\in I_F})+ h(1,(\frac{(a_{\mathbf{i}}\mathbf{u^{i}})'}{a_{\mathbf{i}}\mathbf{u^{i}}})_{{\mathbf{i}}\in I_F})\cr
 & \le|S|+(2|I_{F}|+1)\tilde{h}(F)+3\gen. 
\end{align}
by Lemma \ref{lem:NSgcd_lb} and \eqref{Sestimate}. The assertion
is now clear since $|I_{F}|\le\binom{n+\deg F}{n}$ and $|S|+3\gen\le 3\max\{1,2\gen-2+|S|\}$. 
\end{proof}

\begin{lemma} \label{lem:coprime-irr}For any irreducible $F(\mathbf{x})=\sum_{\mathbf{i}\in I_{F}}a_{\mathbf{i}}\mathbf{x}^{\mathbf{i}}\in K[x_{1},\dots,x_{n}]$
and $\mathbf{u}\in(K^{*})^{n}$, the two polynomials $F$ and $D_{\mathbf{u}}(F)$
are not coprime if and only if $\frac{a_{\mathbf{i}}\mathbf{u}^{\mathbf{i}}}{a_{\mathbf{j}}\mathbf{u}^{\mathbf{j}}}\in{\bf k}^{*}$
whenever $\mathbf{i},\mathbf{j}\in I_{F}$. \end{lemma}

\begin{proof} It is clear from \eqref{DuF} that $\deg D_{\mathbf{u}}(F)\le\deg F$
for each $j$. Since $F$ is irreducible, it follows that $F$ and
$D_{\mathbf{u}}(F)$ are not coprime if and only if $D_{\mathbf{u}}(F)=\lambda F$
for some $\lambda\in K$, i.e., $\frac{(a_{\mathbf{i}}\mathbf{u}^{\mathbf{i}})'}{\mathbf{u}^{\mathbf{i}}}=\lambda a_{\mathbf{i}}$
for each $\mathbf{i}$. The latter condition is equivalent to that
for those $\mathbf{i},\mathbf{j}\in I_{F}$ we must have $\frac{(a_{\mathbf{i}}\mathbf{u}^{\mathbf{i}})'}{a_{\mathbf{i}}\mathbf{u}^{\mathbf{i}}}=\frac{(a_{\mathbf{j}}\mathbf{u}^{\mathbf{j}})'}{a_{\mathbf{j}}\mathbf{u}^{\mathbf{j}}}$,
which is equivalent to that $\left(\frac{a_{\mathbf{i}}\mathbf{u}^{\mathbf{i}}}{a_{\mathbf{j}}\mathbf{u}^{\mathbf{j}}}\right)'=0$.
\end{proof} \begin{lemma} \label{lem:coprime-gen} Let $F=\prod_{i=1}^{r}P_{i}\in K[x_{1},\hdots,x_{n}]$,
where $P_{i}$, $1\le i\le r$, is irreducible and not monomial in
$K[x_{1},\hdots,x_{n}]$. Let $\mathbf{u}\in(K^{*})^{n}$, $\mathbf{e}=(e_{1},\hdots,e_{r})$
be an $r$-tuple of positive integers. Then either the two polynomials
$F$ and 
\[
F_{\mathbf{e},\mathbf{u}}:=\sum_{i=1}^{r}e_{i}D_{\mathbf{u}}(P_{i})\prod_{j\ne i}P_{j}
\]
are coprime in $K[x_{1},\hdots,x_{n}]$ or 
\[
h(u_{1}^{m_{1}}\cdots u_{n}^{m_{n}})\le h(F)
\]
for some $(m_{1},\hdots m_{n})\in\mathbb{Z}^{n}\setminus\{(0,\ldots,0)\}$
with $\sum|m_{i}|\le2\deg F$.

\end{lemma}

\begin{proof} 
If $F$ and $F_{\mathbf{e},{\mathbf{u}}}$ are not coprime in $K[x_{1},\dots,x_{n}]$,
some $P_{i}$ must divide $F_{\mathbf{e},\mathbf{u}}$ and thus divide
$D_{\mathbf{u}}(P_{i})$. Since $P_{i}$ is not a monomial, we have
$P_{i}=\sum_{\mathbf{i}\in I_{P_{i}}}a_{\mathbf{i}}\mathbf{x}^{\mathbf{i}}$
with $|I_{P_{i}}|\ge2$. Then Lemma \ref{lem:coprime-irr} implies
that $\frac{a_{\mathbf{i}}\mathbf{u}^{\mathbf{i}}}{a_{\mathbf{j}}\mathbf{u}^{\mathbf{j}}}\in{\bf k}^{*}$
whenever $\mathbf{i},\mathbf{j}\in I_{P_{i}}$. Since $|I_{P_{i}}|\ge2$,
we can choose distinct $\mathbf{i},\mathbf{j}\in I_{P_{i}}$. Thus
\[
h(\mathbf{u}^{\mathbf{i}-\mathbf{j}})=h(a_{\mathbf{i}}^{-1}a_{\mathbf{j}})\le h(P_{i})\le h(F),
\]
by \eqref{Gaussht} with $0\ne|\mathbf{i}-\mathbf{j}|\le2\deg P_{i}\le2\deg F.$
\end{proof} 
\begin{lemma}\label{dth_power_count} Let $d\ge2$
be an integer, $F_{1},\hdots,F_{r}\in O_{S}[x_{1},\dots,x_{n}]$ be
distinct non-monomial polynomials which are irreducible in $K[x_{1},\dots,x_{n}]$,
and put $F:=F_{1}^{e_{1}}\cdots F_{r}^{e_{r}}$ with $1\le e_{i}<d$
for each $i$. Let $\mathbf{u}=(u_{1},\ldots,u_{n})\in(\mathcal{O}_{S}^{*})^{n}$.  If $F(\mathbf{u})=g^{d}$ for some $g\in K^{*}$,
then for every $\varepsilon>0$ there exist an integer $m$ and reals
$c_{1}$, $c_{2}$, all depending only on {$(\varepsilon,\delta,d)$,
where $\deg F\le\delta$,} such that either 
\begin{equation}
N_{S}(F(\mathbf{u}))\le\varepsilon\max_{1\le j\le n}\{h(u_{j})\}.\label{1}
\end{equation}
or 
\begin{equation}
h(u_{1}^{m_{1}}\cdots u_{n}^{m_{n}})\le c_{1}\tilde{h}(F)+c_{2}\max\{1,2\gen-2+|S|\}\label{2}
\end{equation}
for some integers $m_{1},\hdots,m_{n}$ not all zeros with $\sum|m_{i}|\le2m$.
\end{lemma} 
\begin{proof}By \eqref{value} and the product rule \eqref{product} of $D_{\mathbf{u}}$,
we have 
\begin{align}
dg^{d-1}g'=D_{\mathbf{u}}(F)(\mathbf{u})=(F_{1}^{e_{1}-1}(\mathbf{u})\cdots F_{r}^{e_{r}-1}(\mathbf{u}))F_{\mathbf{e},\mathbf{u}}(\mathbf{u}),\label{expression1}
\end{align}
where  $F_{\mathbf{e},\mathbf{u}}:=\sum_{i=1}^{r}e_{i}D_{\mathbf{u}}(P_{i})\prod_{j\ne i}P_{j}$
as defined in Lemma \ref{lem:coprime-gen}, from which it follows
that either   $\bar{F}:=F_{1}\cdots F_{r}$ and $F_{\mathbf{e},\mathbf{u}}$
are coprime in $K[x_{1},\dots,x_{n}]$ or the second assertion \eqref{2}
holds  with any $(c_{1},c_{2},m)$ with $c_{1}\ge1$,
$c_{2}\ge0$ and $m\ge\d$. It remains to consider the case where
the former condition holds. By Theorem \ref{movinggcdunit} and Proposition \ref{heightDu},
for any $\epsilon'>0$ there exist an
integer $m\ge\d$,  positive reals
$c_{i}$, $1\le i\le4$, with $c_{1}\ge1$ and $c_{2}\ge0$,
depending only on $\epsilon'$, such that whenever 
\begin{align}
\max_{1\le i\le n}h(u_{i})\ge c_{3}\tilde{h}(F)+c_{4}\max\{1,2\gen-2+|S|\},\label{heightpart}
\end{align}
we have either 
\[
h(u_{1}^{m_{1}}\cdots u_{n}^{m_{n}})\le c_{1}\tilde{h}(F)+c_{2}\max\{1,2\gen-2+|S|\}
\]
for some integers $m_{1},\hdots,m_{n}$, not all zeros with $\sum|m_{i}|\le2m$,
or 
\begin{align}
N_{S,{\rm gcd}}(\bar{F}(\mathbf{u}),F_{\mathbf{e},\mathbf{u}}(\mathbf{u}))\le\epsilon'\max_{1\le i\le n}h(u_{i}).\label{gcdpart}
\end{align}
We note that 
$h(u_{1}^{m_{1}}\cdots u_{n}^{m_{n}})\le\left(\sum_{1\le i\le n}|m_{i}|\right)\max_{1\le i\le n}h(u_{i})$,
which shows that the case where \eqref{heightpart} does not hold
 leads to \eqref{2}  once we enlarge 
$c_{1}$ and $c_{2}$.

If \eqref{gcdpart} holds, then together with \eqref{expression1},
we have 
\begin{equation}
 N_{S,\gcd}(F(\mathbf{u}),D_{\mathbf{u}}(F)(\mathbf{u}))\le\sum_{i=1}^{r}(e_{i}-1)N_{S}(F_{i}(\mathbf{u}))+\varepsilon'\max_{1\le j\le n}\{h(u_{j})\}.\label{upbound}
\end{equation}
On the other hand,  since $g^{d}=F(\mathbf{u})$, the
key equality \eqref{value}  and Lemma \ref{lem:NSgcd_lb} imply
that 
\[
 N_{S,\gcd}(F(\mathbf{u}),D_{\mathbf{u}}(F)(\mathbf{u}))=N_{S,\gcd}(g^{d},(g^{d})')\ge(d-1)N_{S}(g)-3\gen;
\]
together
with the fact $F(\mathbf{u}),F_{1}(\mathbf{u}),\ldots,F_{r}(\mathbf{u})\in O_{S}$,
this  gives 
\[
 N_{S,\gcd}(F(\mathbf{u}),D_{\mathbf{u}}(F)(\mathbf{u}))+3\gen\ge\frac{d-1}{d}N_{S}(F(\mathbf{u}))=\frac{d-1}{d}\sum_{i=1}^{r}e_{i}N_{S}(F(\mathbf{u})). 
\]
Together with \eqref{upbound}, we have 
\[
\sum_{i=1}^{k}(1-\frac{e_{i}}{d})N_{S}(F_{i}(\mathbf{u}))\le2\varepsilon'\max_{1\le j\le n}\{h(u_{j})\},
\]
by further requiring that $3\gen\le\varepsilon'\max_{1\le j\le n}\{h(u_{j})\}$;
this is possible since we may assume that $c_{4}\ge3\gen/\epsilon'$
in \eqref{heightpart}. Since $e_{i}<d$ and $N_{S}(F_{i}(\mathbf{u}))\ge0$
for each $i$, it implies that 
 
\[
\frac{1}{d}N_{S}(F_{i}(\mathbf{u}))\le2\varepsilon'\max_{1\le j\le n}\{h(u_{j})\}
\]
for each $i$. By taking $\varepsilon'= \frac{\varepsilon}{2d\d}\le \frac{\varepsilon}{2d\deg F}\le\frac{\varepsilon}{2d(e_{1}+\cdots+e_{r})}$,
we have 
\[
N_{S}(F(\mathbf{u}))=\sum_{i=1}^{r}e_{i}N_{S}(F_{i}(\mathbf{u}))\le\varepsilon\max_{1\le j\le n}\{h(u_{j})\}.
\]
\end{proof}

\section{\label{section4}Proof of Theorem \ref{thm:y^e=00003D00003D00003D00003D00003DFg^l} }

For each finite extension $L$ over $K$, denote by $h_{L}$ the height
function (both on $L$ and on $L[x_{1},\dots,x_{n}]$) obtained from
the same construction of $h$ with $K$ replaced by $L$;  similar
for the notation $\tilde{h}_{L}$, and $O_{L,\widetilde{S}}$, $N_{L,\widetilde{S}}$,
$\overline{N}_{L,\widetilde{S}}$, where $\widetilde{S}\subset C_{L}(\mathbf{k})$
is a finite subset, and $C_{L}$ be a smooth projective
curve over $\mathbf{k}$ such that $L=\mathbf{k}(C_{L})$. We need the following result from \cite[Proposition 2.4]{PW2015}.

\begin{proposition}\label{Prop: genus}

Let $\alpha$ be a nonconstant algebraic element over $K$ with $[K(\alpha):K]=m$.
Denote by $L=K(\alpha)$ and let $C_{L}$ be a smooth projective curve
over $\mathbf{k}$ of genus $\gen_{L}$ such that $L=\mathbf{k}(C_{L})$.
Then

\[
\gen_{L}-1\le m(\gen-1)+(m-1)h_{L}(\alpha).
\]
\end{proposition}

In the following proof, we will use, without further prompts, the
standard fact that $h_{L}(a)=[L:K]h(a)$ for every $a\in K$, and
that $\tilde{h}_{L}(P)\le{[L:K]}\tilde{h}(P)$ for every
$P\in K[x_{1},\dots,x_{n}]$.

\begin{proof}[Proof of Theorem \ref{thm:y^e=00003D00003D00003D00003D00003DFg^l}]
We may assume $|S|\ge2$, for otherwise $\mathcal{O}_{S}^{*}=\mathbf{k}^{*}$
and the desired conclusion holds trivially.  For each
$\ell\in\mathbb{N}$, put $\mathbf{u}^{\ell}:=(u_{1}^{\ell},\ldots,u_{n}^{\ell})\in(O_{S}^{*})^{n}$.
We may suppose that there is indeed some $\ell\in\mathbb{N}$ with
\begin{equation}
\ell\ge c_{1}\tilde{h}(F)+c_{2}\max\{1,2\gen-2+|S|\}\label{eq: large_ell}
\end{equation}
such that $F(\mathbf{u^{\ell}})$ is a $d$-th power in $K$, where
$c_{1}$ and $c_{2}$ will be determined in the end of
the proof.


Fix a total ordering on the set of monomials in $K[x_{1},\dots,x_{n}]$
and say that an element $Q\in K[x_{1},\dots,x_{n}]$ is monic if the
coefficient attached to largest monomial appearing in $Q$ with a
non-zero coefficient is $1$. Since $F\ne0$ in $K[x_{1},\dots,x_{n}]$,
it follows from our hypothesis that we may write $F=a\mathbf{x^{i}}G^{d}P$,
where   $P\in K[x_{1},\dots,x_{n}]\setminus K$  is $d$-th
power free monic polynomial
with no  (non-trivial)  monomial factors,  
$G\in K[x_{1},\dots,x_{n}]$ is monic, $\mathbf{x^{i}}\in K[x_{1},\dots,x_{n}]$ 
is a monomial and $a\in K^{*}$. Note that 
\begin{align}
{\normalcolor }\tilde{h}(P)=h(P)\le h(F)\le\tilde{h}(F),\label{heightFP}
\end{align}
and $h(a)\le\tilde{h}(F)$ since $a$ is the coefficients of the largest
monomial appearing in $F$.

Write 
\begin{align}
{\normalcolor {\normalcolor {\normalcolor }}}P=\sum_{\mathbf{i}\in I_{P}}a_{\mathbf{i}}{\bf x}^{\mathbf{i}}\qquad\text{with each \ensuremath{a_{\mathbf{i}}} being nonzero}.\label{P}
\end{align}
By our setting, we have that $|I_{P}|\ge2$.  Choose
a finite subset $S_{P}\subset C(\mathbf{k})$ containing $S$ such
that $a_{\mathbf{i}}\in O_{S_{P}}^{*}$ for each $\mathbf{i}\in I_{P}$,
that each monic irreducible factor of $P$ is in $O_{S_{P}}[x_{1},\dots,x_{n}]$,
and that by \eqref{heightFP} we have 
\begin{align}
2\le|S_{P}|\le|S|+2|I_{P}|\tilde{h}(P)+\deg P\cdot\binom{n+\deg P}{n}\tilde{h}(P)\le|S|+(\deg F+2)\binom{n+\deg{F}}{n}\tilde{h}(F).\label{SP}
\end{align}

  Let $L:=K(\alpha)$ with some $d$-th root
$\alpha$ of $a\mathbf{u^{i}}$.  Since $F(\mathbf{u^{\ell}})$
is a $d$-th power in $K$, it follows that $P(\mathbf{u^{\ell}})$ 
is a $d$-th power in $L$. By Proposition \ref{Prop: genus}, $L$
is the function field of a smooth projective algebraic curve $C_{L}$
of genus $\mathfrak{g}_{L}$ defined over $\mathbf{k}$ with 
\begin{align}
{\normalcolor {\normalcolor {\normalcolor }}}\gen_{L}-1 & \le[L:K](\gen-1)+([L:K]-1)\frac{[L:K]}{d}h(a\mathbf{u}^{i})\cr
 & \le[L:K]\left(\gen-1+\tilde{h}(F)+\deg F\max_{1\le j\le n}h(u_{j})\right)\label{eq: g_L}
\end{align}
since $[L:K]\le d$ and   $\a^{d}=a\mathbf{u}^{\mathbf{i}}$.
 Let $\widetilde{S_{P}}\subset C_{L}(\mathbf{k})$ be the the preimage
of $S_{P}$ under the natural map $C_{L}(\mathbf{k})\rightarrow C(\mathbf{k})$.
Then 
\begin{align}
2\le|\widetilde{S_{P}}|\le[L:K]|S_{P}|.\label{wildtilde_S}
\end{align}


 Now we have that $P\in L[x_{1},\dots,x_{n}]$
is $d$-th power free and has no (non-trivial) monomial factor, that
each irreducible factor of $P$ is in $O_{L,\widetilde{S_{P}}}[x_{1},\dots,x_{n}]$,
and that $P(\mathbf{u^{\ell}})$ is a $d$-th power in $L$. Hence 
we can apply Lemma \ref{dth_power_count} with 
\begin{align}
\varepsilon=\frac{1}{\binom{n+\deg F}{n}^{2}\max_{1\le j\le n}h(u_{j})}
\end{align}
and obtain an integer $m'$ and constants $c_{1}'$, $c_{2}'$ depending
only on $(\varepsilon,\deg F,d)$ such that  we have
either that
\begin{equation}
{\normalcolor }{\normalcolor {\normalcolor {\normalcolor {\normalcolor }}}}N_{L,\widetilde{S_{P}}}(P(\mathbf{u^{\ell}}))\le\varepsilon\ell\max_{1\le j\le n}h_{L}(u_{j})=\frac{\ell[L:K]}{\binom{n+\deg F}{n}^{2}}\;\label{11}
\end{equation}
or that 
\begin{align}
{\normalcolor {\normalcolor {\normalcolor }}}{\normalcolor {\normalcolor {\normalcolor }}}h_{L}(u_{1}^{\ell m_{1}}\cdots u_{n}^{\ell m_{n}}) & \le c_{1}'\tilde{h}_{L}(P)+c_{2}'\max\{1,2\gen_{L}-2+|\widetilde{S_{P}}|\}\label{12}
\end{align}
 for some integers $m_{1},\hdots,m_{n}$, not all zeros with $\sum|m_{i}|\le2m'$.
  First consider the case where \eqref{12} holds; thus
by \eqref{SP}, \eqref{heightFP}, \eqref{eq: g_L} and \eqref{wildtilde_S},
we obtain 
\[
h(u_{1}^{\ell m_{1}}\cdots u_{n}^{\ell m_{n}})\le\left(c_{1}'+c_{2}'(\deg F+2)\binom{n+\deg{F}}{n}\right)\tilde{h}({F})+c_{2}'\max\left\{ 1,2\gen-2+2\tilde{h}(F)+2\deg F\max_{1\le j\le n}h(u_{j})+|S|\right\} .
\]
If $u_{1}^{m_{1}}\cdots u_{n}^{m_{n}}{\in K\setminus{\bf k}}$, then
$h(u_{1}^{\ell m_{1}}\cdots u_{n}^{\ell m_{n}})\ge\ell$ and we would
get a contradiction to \eqref{12}, provided 
\begin{align}
 \ell>\left(c_{1}'+c_{2}'(\deg F+2)\binom{n+\deg F}{n}+2c_{2}'\right)\tilde{h}(F)+c_{2}'\left(1+2\deg F\max_{1\le j\le n}h(u_{j})\right)\max\{1,2\gen-2+|S|\}.\label{l1}
\end{align}

It remains to consider when \eqref{11} occurs. By \eqref{P}, we
have the following equality 
\begin{equation}
P(\mathbf{u^{\ell}})=\sum_{\mathbf{i}\in I_{P}}a_{\mathbf{i}}\mathbf{u^{i\ell}}.\label{eq:unit-eq2}
\end{equation}
First consider the case where the right-hand side of \eqref{eq:unit-eq2}
has a nontrivial vanishing subsum. (This includes the possibility
where $P(\mathbf{u^{\ell}})=0$.) In this case, it must have a smallest
nontrivial vanishing subsum, i.e., for some $I\subset I_{P}$ (with
$|I|\ge2)$ we have 
\begin{equation}
\sum_{\mathbf{i}\in I}a_{\mathbf{i}}\mathbf{u^{i\ell}}=0.\label{eq:smallest_vanish2}
\end{equation}
Corollary \ref{Green} implies that if 
\begin{align}
\ell & >(|I|-1)^{2}(|I|-2)\max\{1,\gen\}+(|I|-1)^{4}h([a_{\mathbf{i}}]_{\mathbf{i}\in I}),\label{l2}
\end{align}
then $\mathbf{u}^{\mathbf{i}-\mathbf{j}}\in{\bf k}$ for any distinct
$\mathbf{i}$, $\mathbf{j}$ in $I$. Since $|I|\le|I_{P}|\le\binom{n+\deg F}{n}$
and $h([a_{\mathbf{i}}]_{\mathbf{i}\in I})\le h(P)\le\tilde{h}(F)$
as well as $\max\{1,\gen\}\le\max\{1,2\gen-2+|S|\}$, we see that
\eqref{l2} holds if 
\begin{equation}
\ell> \binom{n+\deg{F}}{n}^{4} \big(\tilde{h}(F)+\max\{1,2\gen-2+|S|\} \big). \label{eq:l2}
\end{equation}
We also note that any distinct $\mathbf{i}$, $\mathbf{j}$ in $I$
satisfies that $|\mathbf{i}-\mathbf{j}|\le2\deg F$. This settles
down the current case.

It remains to consider the case where the right-hand side of \eqref{eq:unit-eq2}
has no nontrivial vanishing subsum. Pick some $\mathbf{i}_{0}\in I_{P}$.
This case is equivalent to the one where the left-hand side of 
\begin{equation}
\frac{P(\mathbf{u^{\ell}})}{a_{\mathbf{i}_{0}}\mathbf{u}^{\mathbf{i}_{0}\ell}}-\sum_{\mathbf{i}\in I_{P}\setminus\{\mathbf{i}_{0}\}}\frac{a_{\mathbf{i}}\mathbf{u^{i\ell}}}{a_{\mathbf{i}_{0}}\mathbf{u}^{\mathbf{i}_{0}\ell}}=1\label{eq: uniteq3-1}
\end{equation}
has no nontrivial vanishing subsum.  By \eqref{11},
we see that 
\begin{equation}
{\normalcolor }{\normalcolor {\normalcolor {\normalcolor }}}{\normalcolor {\normalcolor {\normalcolor {\normalcolor }}}}N_{S_{P}}(P(\mathbf{u^{\ell}}))\le\frac{\ell}{\binom{n+\deg F}{n}^{2}}.\label{11-1}
\end{equation}
Let $S_{P,\ell}\subset C(\mathbf{k})$ be a subset containing $S_{P}$
and the zeros of $P(\mathbf{u^{\ell}})$ such that 
\[
2\le|S_{P,\ell}|\le|S_{P}|+\overline{N}_{S_{P}}(P(\mathbf{u^{\ell}}))\le|S|+(\deg F+2)\binom{n+\deg{F}}{n}\tilde{h}(F)+\frac{\ell}{\binom{n+\deg F}{n}^{2}}
\]
by \eqref{SP} and \eqref{11-1}.  Applying Theorem \ref{BrMa}
to the $ S_{P,\ell}$-unit equation \eqref{eq: uniteq3-1},
we see that if $\mathbf{u}^{\mathbf{i}-\mathbf{i_{0}}}\not\in{\bf k}$
for some $\mathbf{i}\in I_{P}\setminus\{\mathbf{i}_{0}\}$, then  
\begin{align*}
 \ell\le h(\mathbf{u}^{(\mathbf{i}-\mathbf{i_{0}})\ell}) & \le h(\frac{a_{\mathbf{i}}\mathbf{u^{i\ell}}}{a_{\mathbf{i}_{0}}\mathbf{u}^{\mathbf{i}_{0}\ell}}){\normalcolor {\normalcolor {\normalcolor }}}+h(a_{\mathbf{i}},a_{\mathbf{i}_{0}})\\
 & \le\frac{|I_{P}|^{2}}{2}\left(2\mathfrak{g}-2+|S_{P,\ell}|\right)+h(P)\\
 & \le\frac{\ell}{2}+\left(\frac{1}{2}\deg F+2\right)\binom{n+\deg F}{n}^{3}\tilde{h}(F)+\frac{1}{2}\binom{n+\deg F}{n}^{2}\max\{1,2\gen-2+|S|\}.
\end{align*}
Then again, $\mathbf{u}^{\mathbf{i}-\mathbf{i_{0}}}\in{\bf k}$ for
every $\mathbf{i}\in I_{P}\setminus\{\mathbf{i}_{0}\}$, where we
note that such $\mathbf{i}$ indeed exists and that $|\mathbf{i}-\mathbf{i}_{0}|\le2\deg F$,
if 
\begin{align}
\ell>\left(\deg F+4\right)\binom{n+\deg F}{n}^{3}\tilde{h}(F)+\binom{n+\deg F}{n}^{2}\max\{1,2\gen-2+|S|\}.\label{l3-1}
\end{align}
  We obtain the desired conclusion by taking $m:=\max\{m',\deg F\}$,
and choose $c_{1}$, $c_{2}$ such that \eqref{eq: large_ell} implies
all of \eqref{l1}, \eqref{eq:l2} and \eqref{l3-1}.
\end{proof}

\section{Proof of Theorem \ref{dPisot}}

\label{sectionPisot}  
We need
the following result from \cite[Proposition 4.2]{NW}, where it is
stated for number fields, but it is clear that the proof works for
any field.

\begin{proposition}\label{moving:Prop} Let $f_{1},f_{2}\in K[x_{0},x_{1},\dots,x_{n}]\setminus K[x_{0}]$
be coprime polynomials. Then, the polynomials $f_{1}(m),f_{2}(m)\in K[x_{1},\dots,x_{n}]$
are coprime for all but perhaps finitely many $m\in\mathbb{N}$. 
\end{proposition}

We also recall the following result of Pasten and the third author
on the generalized B{\"u}chi's $n$-th power problem. \begin{theorem}\cite[Theorem 3]{PW2015}\label{buchi_func_field}
Let $K$ be a function field of a smooth projective curve $C$ of
genus $\gen_{K}$ over an algebraically closed field $k$ of characteristic
zero. Let $n\ge2$ and $M$ be integers with 
\[
M>4n\max\{\gen-1,0\}+11n-3.
\]
Let $F\in K[x]\setminus\mathbf{k}[x]$ be a monic polynomial of degree
$n$. Write $F=PH$ where $P\in\mathbf{k}[x]$ is monic, $H\in K[x]$
is monic and $H$ is not divisible by any non-constant polynomial
in $\mathbf{k}[x]$. Let $G_{1},\dots,G_{\ell}\in K[x]$ be the distinct
monic irreducible factors of $H$ (if any) and let $e_{1},\dots,e_{\ell}\ge1$
be integers such that $H=\prod_{j=1}^{\ell}G_{j}^{e_{j}}$. Let $\mu\ge\max_{j}e_{j}$
be an integer and let $a_{1},\dots,a_{M}$ be distinct elements of
$\mathbf{k}$.

If for each $1\le i\le M$, the zero multiplicity of those nonzero
$F(a_{i})\in K^{*}$ at every point $\mathfrak{p}\in C(\mathbf{k})$
is divisible by $\mu$, then $\mu=e_{1}=\cdots=e_{\ell}$ and $H=(\prod_{j=1}^{\ell}G_{j})^{\mu}$.
\end{theorem}

\begin{proof}[Proof of Theorem \ref{dPisot}] Let $u_{1},\dots,u_{n}$
be a (multiplicative) basis of $\Gamma$. Then there exists a Laurent
polynomial $f\in K[x_{0},x_{1},x_{1}^{-1},\dots,x_{n},x_{n}^{-1}]$
such that 
\begin{equation}
b(m)=f(m,u_{1}^{m},\dots,u_{n}^{m}).\label{bf}
\end{equation}
We may assume that $f\in K[x_{0},x_{1},\dots,x_{n}]$ by multiplying
$f$ by $(x_{1}\cdots x_{n})^{hd}$ for some $h\in\mathbb{N}$ without
affecting the assertion.
To avoid trivialities, we assume $f$ is not the zero polynomial.
We also note that the assumption
that $\Gamma\cap k^{*}=\{1\}$ implies that $u_{1},\dots,u_{n}$ are
multiplicatively independent modulo ${\bf k}$.

For each $m\in\mathbb{N}$, it is clear that $\deg f(m,\bullet)\le\deg f$;
since $v_{\p}(a(m))\ge v_{\p}(a)$ for every $\p\in C(\mathbf{k})$
and nonzero $a\in K[x_{0}]$, we also see that $\tilde{h}(f(m,\bullet))\le\tilde{h}(f)$.
Denote by $\mathcal{N}$ the collection of $m\in\mathbb{N}$ such
that $b(m)$ is a $d$-th power in $K$, which is an infinite set
by the assumption. Thus $f(m,u_{1}^{m},\dots,u_{n}^{m})$ is a $d$-th
power in $K$ for each $m\in\mathcal{N}$. Let $S\subset C(\mathbf{k})$
be a finite subset such that $\mathbf{u}:=(u_{1},\ldots,u_{n})\in(\mathcal{O}_{S}^{*})^{n}$.
Recall $\mathbf{x}:=(x_{1},\ldots,x_{n})$. Applying Theorem \ref{thm:y^e=00003D00003D00003D00003D00003DFg^l}
to each $f(m,\bullet)\in K[\mathbf{x}]$ for each $m\in\mathcal{N}$,
we conclude that 
\begin{align}
f(m,\bullet)=\alpha_{m}\mathbf{x}^{\mathbf{i}_{m}}G_{m}^{d}\label{Qm}
\end{align}
for some $\a_{m}\in K^{*}$, monomial $\mathbf{x}^{\mathbf{i}_{m}}\in K[\mathbf{x}]$
and $G_{m}\in K[\mathbf{x}]$, provided that $m\in\mathcal{N}$ is
sufficiently large.

On the other hand, we factor $f$ in $K[x_{0},\mathbf{x}]$ as 
\begin{align}
{\normalcolor }f(x_{0},\mathbf{x})=Q(x_{0})\mathbf{x}^{\mathbf{i}}\prod_{i=1}^{s}P_{i}(x_{0},\mathbf{x})^{e_{i}}\label{factoringf}
\end{align}
with some $Q\in K[x_{0}]$, some monomial $\mathbf{x^{i}}\in K[\mathbf{x}]$,
and some irreducibles $P_{1},\ldots,P_{s}\in K[x_{0},x_{1},\dots,x_{n}]\setminus K[x_{0}]$
without any (non-trivial) monomial factor. Applying Proposition \ref{moving:Prop}
to all $(P_{i},P_{j})$ with $0\le i<j\le s$, where $P_{0}:=x_{1}x_{2}\cdots x_{n}$,
we may replace $\mathcal{N}$ by one of its cofinite subset such that
for all $m\in\mathcal{N}$, $1\le i\le s$ and $1\le j\le s$ with $i\ne j$, we have that $P_{i}(m,\bullet)\in K[\mathbf{x}]$ neither belongs to $K$ nor
has  (nontrivial) monomial factor, and that  $P_{i}(m,\bullet),P_{j}(m,\bullet)$
share no irreducible factor in $K[\mathbf{x}]$. Since each such $P_{i}(m,\bullet)\in K[\mathbf{x}]$
has at least one irreducible factor, by comparing \eqref{factoringf}
with \eqref{Qm}, we see that each $e_{i}$ must be divisible by $d$,
and thus 
\begin{equation}
f(x_{0},\mathbf{x})=Q(x_{0})\mathbf{x}^{\mathbf{i}}G(x_{0},\mathbf{x})^{d}\label{eq: f}
\end{equation}
 for some $G\in K[x_{0},\mathbf{x}]$.  Letting $\b\in K^{*}$ be the leading
coefficient of $Q$, we have following factorization
\begin{equation}
Q=  \b  Q_0 Q_1^d, \label{eq: factor_Q}
\end{equation}
where $Q_0, Q_1\in K[x_0]$ is monic such  that $Q_0$ is $d$-th power free in $K[x_0]$.
Choose $\gamma_1,\gamma_2\in\overline{K}$ such that
\begin{equation}
\gamma_1^{d}=\b\qquad\text{and}\qquad\gamma_2^{d}=\mathbf{u}^{\mathbf{i}}.\label{eq: ga}
\end{equation}
By \eqref{bf},
\eqref{eq: f}, \eqref{eq: factor_Q} and \eqref{eq: ga}, we see that 
\[
Q_0(m)=b(m)\left((\g_1\g_2 ^{m} )^{d} Q_1(m)^d  G(m,u_{1}^{m},\dots,u_{n}^{m})^{d}\right)^{-1}
\]
is a $d$-th power in the function field $K(\g_1,\g_2)$ over $\mathbf{k}$
for these infinitely many $m\in\mathcal{N}$. Now Theorem \ref{buchi_func_field}
implies that $Q_0\in\mathbf{k}[x_{0}]$. Therefore, our desired
conclusion holds with $ R:=Q_0$ and $a$ given
by $m\mapsto\g_1\gamma_2^{m}Q_1(m)G(m,u_{1}^{m},\dots,u_{n}^{m}).$
\end{proof}

\section{Proof of the GCD Theorems}

\label{gcd}

\subsection{Key Theorems}

We first recall some definitions in order to reformulate Theorem 2.2
in \cite{Wa2004}, which deal with the case when the coefficients
of the linear forms are in $\K$ instead of constants, i.e., in $\mathbf{k}$.
Consider $q$ (nonzero) linear forms $L_{j}:=a_{j0}X_{0}+\dots+a_{jn}X_{n}$,
$1\le j\le q$, with each $a_{jk}$ in $K$. Recall that the Weil
function associated with $L_{j}$ at a place $\p$ of $K$ is defined
by sending those $\mathbf{a}\in\mathbb{P}^{n}(\K)$ with $L_{j}(\mathbf{a})\ne0$
to

\[
\lambda_{L_{j},\p}(\mathbf{a}):=v_{\p}(L_{j}(\mathbf{a}))-v_{\p}(\mathbf{a})-v_{\p}(L_{j}).
\]

For any finite-dimensional vector subspace %
\mbox{%
$V\subset K$%
} over %
\mbox{%
$\mathbf{k}$%
} and any positive integer %
\mbox{%
$r$%
}, we denote by %
\mbox{%
$V(r)$%
} the vector space over %
\mbox{%
$\mathbf{k}$%
} spanned by the set of all products of %
\mbox{%
$r$%
} (non-necessarily distinct) elements from %
\mbox{%
$V$%
}. It is easy to show (e.g., %
\mbox{%
\cite[Lemma 6]{Wa1996}%
}) that %
\mbox{%
$\dim V(r+1)\ge\dim V(r)$%
} for each %
\mbox{%
$r$%
} and %
\mbox{%
$\liminf_{r\to\infty}\dim V(r+1)/\dim V(r)=1$%
}. Applying this inequality with %
\mbox{%
$V$%
} replaced by %
\mbox{%
$V(e)$%
}, we see that for each %
\mbox{%
$e\in\mathbb{N}$%
} 
\begin{align}
\liminf_{r\to\infty}\dim V(er+e)/\dim V(er)=1.\label{infvr}
\end{align}

\begin{definition}
Let $E\subset\K$ be a vector space over ${\bf k}$.  We say that $y_{1},\ldots,y_{m}\in K$ are linearly
nondegenerate over $E$ if whenever we have a linear combination $\sum_{i=1}^{m}a_{i}y_{i}=0$
with $a_{i}\in E$, then $a_{i}=0$ for each $i$; otherwise we say
that they are linearly degenerate over $E$. Similarly, a point ${\bf x}=[x_{0}:x_{1}:\cdots:x_{n}]\in\mathbb{P}^{n}(\K)$,
with each $x_{i}\in K$, is said to be linearly degenerate (resp.
linearly nondegenerate) over $E$ if $x_{0},\ldots,x_{0}$ is linearly
degenerate (resp. nondegenerate) over $E$.\end{definition}

We obtain the following variant of Theorem 2.2 in \cite{Wa2004} from
its proof.

\begin{theorem}\label{MSMT} 
Consider the
collection $\mathcal{L}:=\{L_{1},\hdots,L_{q}\}$ of linear forms
$L_{i}=\sum_{j=0}^{n}a_{ij}X_{j}\in\K[X_{0},\hdots,X_{n}]$, $1\le i\le q$,
and define 
\begin{equation}
h(\mathcal{L}):=-\sum_{\mathbf{p}\in C(\mathbf{k})}\min_{1\le i\le q,\,0\le j\le n}v_{\mathbf{p}}(a_{ij}).\label{eq: new_h(H)-1}
\end{equation}
Let $V_{\mathcal{L}}\subset K$ be the vector subspace over $\mathbf{k}$
spanned by the set consisting of all the $a_{ij}$. Suppose that ${\bf a}\in\mathbb{P}^{n}(\K)$
is linearly nondegenerate over $V_{\mathcal{L}}(r+1)$ for some positive
integer $r$, then 
\begin{align}
\sum_{\p\in S}\max_{J}\sum_{j\in J}\lambda_{L_{j},\p}({\bf a})\le\frac{w}{u}(n+1)\left(h({\bf a})+(r+2)h(\mathcal{L})+\frac{nw+w-1}{2}\max\{0,2\gen-2+|S|\}\right),
\end{align}
where the maximum is taken over all subsets $J\subset\{1,\hdots,q\}$
such that those linear forms $L_{j}$ with $j\in J$ are linearly
independent over $\K$, and we denote by $w:=\dim_{\mathbf{k}}V_{\mathcal{L}}(r+1)$
and $u:=\dim_{\mathbf{k}}V_{\mathcal{L}}(r)$. \end{theorem}

We now formulate the following technical theorem of estimating the
counting function of the gcd. The proof is adapted from \cite{Levin:GCD}
and \cite{levin2019greatest} with more control on the coefficients
of the constructed linear forms so that all the constants involved
can be computed effectively.

\begin{theorem}\label{Refinement} Let $F_{1},F_{2}\in K[x_{1},\cdots,x_{n}]$
be coprime polynomials of the same degree $d>0$. Assume that one
of the coefficients in the expansion of $F_{i}$ is 1 for each $i\in\{1,2\}$.
For every positive integer $m\ge2d$, we let $M:=M_{m}:=2\binom{m+n-d}{n}-\binom{m+n-2d}{n}$
and $M':=M'_{m}:=\binom{m+n}{n}-M$. For every positive integer $r$,
we denote by $V_{F_{1},F_{2}}(r)$ the (finite-dimensional) vector
space over ${\bf k}$ spanned by $\prod_{\alpha}\alpha^{n_{\alpha}}$,
where $\alpha$ runs over all non-zero coefficients of $F_{1}$ and
$F_{2}$, $n_{\alpha}\ge0$ and $\sum n_{\alpha}=r$; we also put
$d_{r}:=\dim_{{\bf k}}V_{F_{1},F_{2}}(r)$. Then $M'_{m}$ has order
$O(m^{n-2})$; moreover, if for some ${\bf g}=(g_{1},\dots,g_{n})\in({\cal O}_{S}^{*})^{n}$ those $\mathbf{g}^{\mathbf{i}}$ with $\mathbf{i}\in\mathbb{Z}_{\ge0}^{n}$
and $|\mathbf{i}|\le m$ are
  linearly nondegenerate over $V_{F_{1},F_{2}}(Mr+1)$ for some positive
integer $m\ge2d$, then we have the following estimate 
\begin{align*}
 & MN_{S,{\rm gcd}}(F_{1}({\bf g}),F_{2}({\bf g}))\\
\le & \left(M'+\frac{d_{Mr}}{d_{M(r-1)}}M-M\right)mn\max_{1\le i\le n}h(g_{i})+cM\left(h(F_{1})+h(F_{2})\right)+c'M\max\{0,2\gen-2+|S|\},
\end{align*}
where $c:=\frac{d_{Mr}}{d_{M(r-1)}}(1+M(r+1))$ and $c':=\frac{d_{Mr}^{2}M}{2d_{M(r-1)}}$.
 \end{theorem}

\begin{proof} We first make some convenient settings. We denote by
$\mathbf{x}:=(x_{1},\ldots,x_{n})$ an $n$-tuple of $n$ (algebraically
independent) variables. Let $m$ be a positive integer. For a subset
$T\subset K[\mathbf{x}]$, we let 
\[
T_{m}=\{f\in T\,:\deg f\le m\}.
\]
By the assumption that one of the coefficients in the expansion of
$F_{i}$ is $1$ for each $i\in\{1,2\}$, we note that

\begin{equation}
v_{\p}(F_{i})\le0\qquad\text{for every }i\in\{1,2\}\text{ and }\p\in C(\mathbf{k}).\label{eq: vF<=00003D00003D00003D00003D0}
\end{equation}

Consider the ideal $(F_{1},F_{2})\subset K[\mathbf{x}]$. If $(F_{1},F_{2})=(1)=K[\mathbf{x}]$,
then it is elementary to show that $N_{S,{\rm gcd}}(F_{1}({\bf g}),F_{2}({\bf g}))$
is bounded by some constant independent of ${\bf g}$. Therefore,
we assume that the ideal $(F_{1},F_{2})$ is proper. For ${\bf i}=(i_{1},\cdots,i_{n})\in\mathbb{Z}_{\ge0}^{n}$,
we let ${\bf x}^{{\bf i}}:=x_{1}^{i_{1}}\cdots x_{n}^{i_{n}}$, and
${\bf g}^{{\bf i}}:=g_{1}^{i_{1}}\cdots g_{n}^{i_{n}}$. By Lemma
2.11 of \cite{levin2019greatest}, we note that $M=\dim_{K}(F_{1},F_{2})_{m}$,
and may choose a basis $\{\phi_{1},\hdots,\phi_{M}\}$ of the $K$-vector
space $(F_{1},F_{2})_{m}$ such that each $\phi_{j}$ is of the form
${\bf x}^{{\bf i}}F_{j}$ with $|\mathbf{i}|:=i_{1}+\cdots+i_{n}\le m-d$
and $j\in\{1,2\}$. Put 
\[
\Phi:=(\phi_{1},\hdots,\phi_{M})\qquad\text{and}\qquad\Phi({\bf g}):=(\phi_{1}({\bf g}),\hdots,\phi_{M}({\bf g})).
\]

For each $\p\in S$, we construct a subset $B_{\p}\subset K[\mathbf{x}]_{m}$,
consisting of only monomials, whose images in the $K$-linear space
$V_{m}:=K[\mathbf{x}]_{m}/(F_{1},F_{2})_{m}$ form one of its bases
as follows. 
Choose a monomial ${\bf x}^{{\bf i}_{\p,1}}\in K[x_{1},\hdots,x_{n}]_{m}$
so that $v_{\p}({\bf g}^{{\bf i}_{\p,1}})$ is maximum subject to
the condition ${\bf x}^{{\bf i}_{\p,1}}\notin(F_{1},F_{2}).$ If ${\bf x}^{{\bf i}_{\p,1}},\hdots,{\bf x}^{{\bf i}_{\p,j}}$
have been constructed such that their images in $V_{m}$ are $K$-linearly
independent but don't span the whole $V_{m}$, then we let ${\bf x}^{{\bf i}_{\p,j+1}}\in K[x_{1},\hdots,x_{n}]_{m}$
be a monomial such that $v_{\p}({\bf g}^{{\bf i}_{\p,j+1}})$ is maximum
subject to the condition that the images of ${\bf x}^{{\bf i}_{\p,1}},\hdots,{\bf x}^{{\bf i}_{\p,j+1}}$
in $V_{m}$ are $K$-linearly independent; otherwise we stop. Because
$\dim_{K}V_{m}=\dim_{K}K[\mathbf{x}]_{m}-\dim_{K}(F_{1},F_{2})_{m}=\binom{m+n}{n}-M=M'$,
we will eventually stop and obtain that $B_{\p}:=\{{\bf x}^{{\bf i}_{\p,1}},\hdots,{\bf x}^{{\bf i}_{\p,M'}}\}\subset K[\mathbf{x}]_{m}$
is a set of monomials whose images in $V_{m}$ form one of its $K$-linear
bases such that 
\begin{equation}
v_{\p}({\bf g}^{{\bf i}_{\p,1}})\ge v_{\p}({\bf g}^{{\bf i}_{\p,2}})\ge\cdots\ge v_{\p}({\bf g}^{{\bf i}_{\p,M'}})\ge v_{\p}({\bf g}^{{\bf i}_{\p}(i)})\label{eq: decreasing-v}
\end{equation}
for each $i\in\{1,\ldots,M\}$, where we denote by $\{{\bf i}_{\p}(1),\hdots,{\bf i}_{\p}(M)\}$
the set of those ${\bf i}\in\mathbb{Z}_{\ge0}^{n}$ with $|{\bf i}|\le m$
and $\mathbf{i}\notin I_{\p}$, where 
\begin{equation}
I_{\p}:=\{{\bf i}_{\p,1},\hdots,{\bf i}_{\p,M'}\}.\label{eq: I_p}
\end{equation}

By direction calculation, we find that $M'_{m}=\binom{m+n}{n}-2\binom{m+n-d}{n}+\binom{m+n-2d}{n}=O(m^{n-2})$.
Alternatively, since $F_{1}$ and $F_{2}$ are coprime, the ideal
$(F,G)$ defines a closed subset of $\mathbb{P}^{n}$ of codimension
at least 2, and it follows from the theory of Hilbert functions and
Hilbert polynomials that $M'_{m}=\dim_{K}V_{m}=O(m^{n-2})$.

For each $i\in\{1,\ldots,M\}$, we have 
\[
{\bf x}^{{\bf i_{\p}}(i)}+\sum_{j=1}^{M'}c_{\p,i,j}{\bf x}^{{\bf i}_{\p,j}}\in(F_{1},F_{2})_{m}
\]
for some (unique) choice of coefficients $c_{\p,i,j}\in\K$; by expressing
${\bf x}^{{\bf i_{\p}}(i)}+\sum_{j=1}^{M'}c_{\p,i,j}{\bf x}^{{\bf i}_{\p j}}$
as a (unique) $K$-linear combination of $\phi_{1},\hdots,\phi_{M}$,
we let 
\begin{equation}
L_{\p,i}:=\sum_{\ell=1}^{M}b_{\p,i,\ell}y_{\ell}\in K[y_{1},\ldots,y_{M}]\label{eq:L_p,i}
\end{equation}
be a linear form over $K$ such that 
\begin{align}
L_{\p,i}(\Phi(\mathbf{x}))=c_{\p}\left({\bf x}^{{\bf i_{\p}}(i)}+\sum_{j=1}^{M'}c_{\p,i,j}{\bf x}^{{\bf i}_{\p,j}}\right),\label{cij}
\end{align}
where $c_{\p}\in K^{*}$ will be chosen later.

By the choice of the $\phi_{\ell}$, we may write 
\begin{align}
\phi_{\ell}=\sum_{s=1}^{M}\alpha_{\p,\ell,s}{\bf x}^{{\bf i}_{\p}(s)}+\sum_{j=1}^{M'}\alpha_{\p,\ell,{\bf i}_{\p,j}}{\bf x}^{{\bf i}_{\p,j}},\label{cij3}
\end{align}
where both $\alpha_{\p,\ell,i}$ and $\alpha_{\p,\ell,{\bf i}_{\p,j}}$
are coefficients of either $F_{1}$ or $F_{2}$, thus 
\begin{equation}
\min\{v_{\p}(\alpha_{\p,\ell,{\bf i}_{\p,j}}),v_{\p}(\alpha_{\p,\ell,i})\}\ge v_{\p}(F_{1})+v_{\p}(F_{2})\qquad\text{for each }\ell,i,j.\label{eq: v_=00003D00003D00003D00005Calpha}
\end{equation}
Combining \eqref{eq:L_p,i} and \eqref{cij3}, we have 
\begin{align}
L_{\p,i}(\Phi(\mathbf{x}))=\sum_{\ell=1}^{M}b_{\p,i,\ell}\left(\sum_{s=1}^{M}\alpha_{\p,\ell,s}{\bf x}^{{\bf i}_{\p}(s)}+\sum_{j=1}^{M'}\alpha_{\p,\ell,{\bf i}_{\p,j}}{\bf x}^{{\bf i}_{\p,j}}\right).\label{cij4}
\end{align}
Note that if we take $c_{\p}=1$, then by comparing \eqref{cij} with
\eqref{cij4}, we find that 
\[
\det(b_{\p,i,\ell})_{1\le\ell,i\le M}\det(\alpha_{\p,\ell,s})_{1\le\ell,s\le M}=1.
\]
From now on, we let 
\begin{equation}
c_{\p}:=\det(\alpha_{\p,\ell,s})_{1\le\ell,s\le M}\ne0\label{eq: c_p}
\end{equation}
and note that $c_{\p}\in$$V_{F_{1},F_{2}}(M)$. With this choice
of $c_{\p}$, we compare \eqref{cij} with \eqref{cij4} again and
see that the inverse of $(\alpha_{\p,\ell,s})_{1\le\ell,s\le M}$
is $c_{\p}^{-1}(b_{\p,i,\ell})_{1\le\ell,i\le M}$, which shows that
\begin{equation}
b_{\p,i,\ell}\in V_{F_{1},F_{2}}(M-1)\label{eq: b_pil_in}
\end{equation}
for each $i$, $\ell$ by Cramer's rule. This comparison also gives
\begin{align}
c_{\p} & =\sum_{\ell=1}^{M}b_{\p,i,\ell}\alpha_{\p,\ell,i}\qquad\text{for each }1\le i\le M,\label{eq: key2}\\
c_{\p}c_{\p,i,j} & =\sum_{\ell=1}^{M}b_{\p,i,\ell}\alpha_{\p,\ell,{\bf i}_{\p,j}}\qquad\text{for each }1\le i\le M\text{ and }1\le j\le M'.\label{eq: key2-1}
\end{align}
From \eqref{eq: decreasing-v}, \eqref{cij}, \eqref{eq: key2}, \eqref{eq: key2-1},
\eqref{eq: v_=00003D00003D00003D00005Calpha} and \eqref{eq: vF<=00003D00003D00003D00003D0},
we have

\begin{align}
v_{\p}(L_{\p,i}(\Phi({\bf g}))) & \ge v_{\p}({\bf g}^{{\bf i}_{\p}(i)})+\min_{j}\{v_{\p}(c_{\p}),v_{\p}(c_{\p}c_{\p,i,j})\}\label{keyinequality2}\\
 & \ge v_{\p}({\bf g}^{{\bf i}_{\p}(i)})+\min_{j}\{\min_{\ell}v_{\p}(b_{\p,i,\ell})+\min_{\ell}v_{\p}(\alpha_{\p,\ell,i}),\min_{\ell}v_{\p}(b_{\p,i,\ell})+\min_{\ell}v_{\p}(\alpha_{\p,\ell,{\bf i}_{\p,j}})\}\nonumber \\
 & \ge v_{\p}({\bf g}^{{\bf i}_{\p}(i)})+\min_{\ell}v_{\p}(b_{\p,i,\ell})+v_{\p}(F_{1})+v_{\p}(F_{2}),\nonumber 
\end{align}
which gives the following key inequality 
\begin{equation}
v_{\p}(L_{\p,i}(\Phi({\bf g})))-v_{\p}(L_{\p,i})\ge v_{\p}({\bf g}^{{\bf i}_{\p}(i)})+v_{\p}(F_{1})+v_{\p}(F_{2}).\label{eq:keyineq}
\end{equation}

Thus, by the construction of \eqref{eq: I_p} and the fact that ${\bf g}^{{\bf i}}\in\mathcal{O}_{S}^{*}$
for each $\mathbf{i}$, we have 
\begin{align}
\sum_{\p\in S}\sum_{1\le i\le M}v_{\p}({\bf g}^{{\bf i}_{\p}(i)}) & =\sum_{\p\in S}\sum_{|{\bf i}|\le m}v_{\p}({\bf g}^{{\bf i}})-\sum_{\p\in S}\sum_{|{\bf i}|\le m,{\bf i}\in I_{\p}}v_{\p}({\bf g}^{{\bf i}})\label{eq:main_term}\\
 & \ge\sum_{|{\bf i}|\le m}\sum_{\p\in S}v_{\p}({\bf g}^{{\bf i}})-|I_{\p}|m\sum_{\p\in S}\sum_{j=1}^{n}v_{\p}^{0}(g_{j})\nonumber \\
 & =-M'm\sum_{j=1}^{n}h(g_{j})\nonumber \\
 & \ge-M'mn\max_{1\le j\le n}h(g_{j}).\nonumber 
\end{align}
By the choice of these $\phi_{i}\in\K[x_{1},\hdots,x_{n}]_{m}$, together
with \eqref{eq: vF<=00003D00003D00003D00003D0}, we have 
\begin{align*}
v_{\p}(\phi_{i}({\bf g})) & \ge m\min\{v_{\p}({\bf g}),0\}+v_{\p}(F_{1})+v_{\p}(F_{2})\\
 & \ge-m\sum_{j=1}^{n}v_{\p}^{\infty}(g_{j})+v_{\p}(F_{1})+v_{\p}(F_{2})
\end{align*}
for every $\p\in C(\mathbf{k})$. It follows that

\begin{align}
h(\Phi({\bf g}))\le mn\max_{1\le i\le n}h(g_{i})+h(F_{1})+h(F_{2}).\label{phiup}
\end{align}
Also, with the fact that ${\bf g}^{{\bf i}}\in\mathcal{O}_{S}^{*}$
for each $\mathbf{i}$, we note for every $\p\notin S$ that $v_{\p}(\phi_{i}({\bf g}))=v_{\p}(F_{\epsilon_{i}}({\bf g}))\ge v_{\p}(F_{\epsilon_{i}})$
with some $\epsilon_{i}\in\{1,2\}$, and hence, together with \eqref{eq: vF<=00003D00003D00003D00003D0},
we have that 
\begin{equation}
\begin{split}v_{\p}(\phi_{i}({\bf g})) & \ge\min\{v_{\p}(F_{1}({\bf g})),v_{\p}(F_{2}({\bf g}))\}\\
 & \ge\min\{v_{\p}^{0}(F_{1}({\bf g})),v_{\p}^{0}(F_{2}({\bf g}))\}+v_{\p}(F_{1})+v_{\p}(F_{2}).
\end{split}
\label{phinotp}
\end{equation}
By \eqref{phinotp}, we have 
\begin{align}
N_{S,{\rm gcd}}(F_{1}({\bf g}),F_{2}({\bf g})) & =\sum_{\p\notin S}\min\{v_{\p}^{0}(F_{1}({\bf g})),v_{\p}^{0}(F_{2}({\bf g}))\}\nonumber \\
 & \le\sum_{\p\notin S}\min_{i}v_{\p}(\phi_{i}({\bf g}))+\sum_{\p\notin S}-(v_{\p}(F_{1})+v_{\p}(F_{2}))\nonumber \\
 & =\sum_{\p\not\in S}v_{\p}(\Phi({\bf g}))+\sum_{\p\notin S}-(v_{\p}(F_{1})+v_{\p}(F_{2}))\nonumber \\
 & =-h(\Phi({\bf g}))-\sum_{\p\in S}v_{\p}(\Phi({\bf g}))+\sum_{\p\notin S}-(v_{\p}(F_{1})+v_{\p}(F_{2})).\label{phgcd}
\end{align}

By \eqref{eq: b_pil_in}, we may choose a finite collection %
\mbox{%
$\mathcal{L}$%
} of linear forms over %
\mbox{%
$K$%
} such that %
\mbox{%
$L_{\p,i}\in\mathcal{L}$%
} for each %
\mbox{%
$\p\in S$%
} and %
\mbox{%
$i\in\{1,\ldots M\}$%
}, that the finite-dimensional %
\mbox{%
$\mathbf{k}$%
}-linear subspace %
\mbox{%
$V:=V_{F_{1},F_{2}}(M)$%
} is spanned by the set of all coefficients of linear forms in %
\mbox{%
$\mathcal{L}$%
}, and that\\
 \raisebox{-\belowdisplayshortskip}{%
\noindent\parbox[b]{1\linewidth}{%
\begin{equation}
h(\mathcal{L})\le M(h(F_{1})+h(F_{2})).\label{eq: h(L)}
\end{equation}
}}

\noindent Since 
those $\mathbf{g}^{\mathbf{i}}$ with $\mathbf{i}\in\mathbb{Z}_{\ge0}^{n}$
and $|\mathbf{i}|\le m$ are
 linearly nondegenerate over $V_{F_{1},F_{2}}(Mr+1)$, we must have
that $\Phi({\bf g})\in\mathbb{P}^{M-1}(K)$ is linearly nondegenerate
over $V_{F_{1},F_{2}}(Mr)=V(r)$. By \eqref{eq: c_p}, we note that
elements from $\{L_{\p,i}(\Phi(\mathbf{x}))\,|\,1\le i\le M\}$ are
linearly independent over $\K$; thus the linear forms $L_{\p,i}$,
$1\le i\le M$, are linearly independent over $K$. Noting that $d_{Mr}=\dim_{\mathbf{k}}V(r)$
and $d_{M(r-1)}:=\dim_{\mathbf{k}}V(r-1)$, we obtain from Theorem
\ref{MSMT} and \eqref{eq: h(L)} that 
\begin{align}
\sum_{\p\in S}\sum_{1\le i\le M}\lambda_{L_{\p,i},\p}(\Phi({\bf g})) & \le\frac{d_{Mr}M}{d_{M(r-1)}}\left(h(\Phi({\bf g)})+(r+1)M(h(F_{1})+h(F_{2}))+\frac{Md_{Mr}-1}{2}\max\{0,2\gen-2+|S|\}\right).\label{eq: MSMT}
\end{align}
Together with \eqref{eq:keyineq}, \eqref{eq:main_term} and \eqref{phgcd},
we have 
\begin{equation}
\begin{split}\sum_{\p\in S}\sum_{1\le i\le M}\lambda_{L_{\p,i},\p}(\Phi({\bf g})) & =\sum_{\p\in S}\sum_{1\le i\le M}\left(v_{\p}(L_{\p,i}(\Phi({\bf g})))-v_{\p}(L_{\p,i})\right)-M\sum_{\p\in S}v_{\p}(\Phi({\bf g}))\\
 & \ge\sum_{\p\in S}\sum_{1\le i\le M}v_{\p}({\bf g}^{{\bf i}_{\p}(i)})+M\sum_{\p\in S}\left(v_{\p}(F_{1})+v_{\p}(F_{2})\right)+MN_{S,{\rm gcd}}(F_{1}({\bf g}),F_{2}({\bf g}))\\
 & \qquad+Mh(\Phi({\bf g)})+M\sum_{\p\notin S}(v_{\p}(F_{1})+v_{\p}(F_{2}))\\
 & \ge-M'mn\max_{1\le j\le n}h(g_{j})+MN_{S,{\rm gcd}}(F_{1}({\bf g}),F_{2}({\bf g}))+M(h(\Phi({\bf g)})-h(F_{1})-h(F_{2})).
\end{split}
\label{eq:MSMT-LHS}
\end{equation}
Combining \eqref{eq:MSMT-LHS} with \eqref{eq: MSMT} and \eqref{phiup},
we get

\[
\begin{split} & MN_{S,{\rm gcd}}(F_{1}({\bf g}),F_{2}({\bf g}))\label{useSMT2}\\
\le & M'mn\max_{1\le j\le n}h(g_{j})+M\left(\frac{d_{Mr}}{d_{M(r-1)}}-1\right)h(\Phi({\bf g)})+\left(\frac{d_{Mr}}{d_{M(r-1)}}M^{2}(r+1)+M\right)\left(h(F_{1})+h(F_{2})\right)\\
 & +\frac{d_{Mr}M(Md_{Mr}-1)}{2d_{M(r-1)}}\max\{0,2\gen-2+|S|\}\\
\le & \left(M'+\frac{d_{Mr}}{d_{M(r-1)}}M-M\right)mn\max_{1\le i\le n}h(g_{i})+\left(\frac{d_{Mr}}{d_{M(r-1)}}(1+M(r+1))\right)M\left(h(F_{1})+h(F_{2})\right)\\
 & +\frac{d_{Mr}^{2}M^{2}}{2d_{M(r-1)}}\max\{0,2\gen-2+|S|\}.
\end{split}
\]

\end{proof}

\begin{theorem}\label{=00003D00003D00003D000024S=00003D00003D00003D000024 part}
Let $F\in K[x_{1},\cdots,x_{n}]$ be a polynomials of degree $d>0$
that does not vanish at $(0,\hdots,0)$. Assume that one of the coefficients
of $F$ is 1. For each %
\mbox{%
$r\in\mathbb{N}$%
}, denote by $V_{F}(r)$ the (finite-dimensional) vector space over
${\bf k}$ spanned by $\prod_{\alpha}\alpha^{n_{\alpha}}$, where
$\alpha$ runs over all (non-zero) coefficients of $F$ with $n_{\alpha}\ge0$
and $\sum n_{\alpha}=r$; put $d_{r}:=\dim_{{\bf k}}V_{F}(r)$. Put
$N:=\binom{n+d}{n}-1$. Let ${\bf g}=(g_{1},\dots,g_{n})\in({\cal O}_{S}^{*})^{n}$.
Suppose 
those $\mathbf{g}^{\mathbf{i}}$ with $\mathbf{i}\in\mathbb{Z}_{\ge0}^{n}$
and $|\mathbf{i}|\le d$ are linearly nondegenerate over $V_{F}(r)$. Then we have the following
estimate 
\begin{align*}
\sum_{\p\in S}v_{\p}^{0}(F({\bf g}))\le & (\frac{d_{r}}{d_{r-1}}-1)(N+1)dn\max_{1\le i\le n}h(g_{i})+\frac{d_{r}(N+1)}{d_{r-1}}(r+1)h(F)\\
 & +\frac{d_{r}(N+1)(Nd_{r}+d_{r}-1)}{2d_{r-1}}\max\{0,2\gen-2+|S|\}.
\end{align*}

\end{theorem}

\begin{proof} Let $\Phi=(\phi_{0},\phi_{1},\hdots,\phi_{N}):\mathbb{P}^{n}\to\mathbb{P}^{N}$
be the $d$-tuple embedding of $\mathbb{P}^{n}$ given by the set
of monomials of degree $d$ in $\K[x_{0},\hdots,x_{n}]$, where $\phi_{0}:=x_{0}^{d}$.
Let $\tilde{F}\in\K[x_{0},\hdots,x_{n}]$ be the homogenization of
$F$. Denote by $\tilde{{\bf g}}:=(g_{0},g_{1},\hdots,g_{n})$, where
$g_{0}:=1$.

Since each $\phi_{i}$ is a degree-$d$ monomial in $\K[x_{0},\hdots,x_{n}]$,
we have $v_{\p}(\phi_{i}(\tilde{{\bf g}}))\ge dv_{\p}(\tilde{{\bf g}})=d\min\{v_{\p}({\bf g}),0\}\ge-d\sum_{j=1}^{n}v_{\p}^{\infty}(g_{j})$
for every $\p\in C(\mathbf{k})$ and $i\in\{0,\ldots,N\}$; thus we
have 
\begin{align}
h(\Phi(\tilde{{\bf g}}))\le dn\max_{1\le i\le n}h(g_{i}).\label{phiup-1}
\end{align}
Also, since $\phi_{i}(\tilde{{\bf g}})\in\mathcal{O}_{S}^{*}$, we
have that 
\begin{equation}
\sum_{\p\in S}v_{\p}(\phi_{i}(\tilde{{\bf g}}))=0.\label{eq: sum_formula_Sunit}
\end{equation}

For each $i\in\{0,\ldots,N\}$, denote by $L_{i}$ the linear form
corresponding to the coordinate hyperplanes $\mathbb{P}^{N}$. We
also denote by $L_{\tilde{F}}\in K[y_{0},\ldots,y_{N}]$ the linear
form coming from the monomial expansion (of degree $d$) of $\tilde{F}$;
thus $L_{\tilde{F}}(\Phi(\tilde{\mathbf{g}}))=\tilde{F}(\tilde{\mathbf{g}})$.
Let $\mathcal{L}:=\{L_{i}\,|\,i\in\{0,\ldots,N\}\}\cup\{L_{\tilde{F}}\}$.
By construction and our assumption that one of the coefficients of
$F$ is 1, we have that $h(\mathcal{L})=\tilde{h}(F)=h(F)$, that
$V_{F}(1)$ is spanned by the set of all coefficients of linear forms
in $\mathcal{L}$, and that 
\begin{equation}
v_{\p}(L)\le0\qquad\text{for every }L\in\mathcal{L}\text{ and }\p\in C(\mathbf{k}).\label{eq: vL<=00003D00003D00003D0}
\end{equation}
We also note that any $N+1$ linear forms in $\mathcal{L}$ are linearly
independent over $K$ since $F(0,\hdots,0)\ne0$.

For those $\p\in S$ and $i\in\{0,\ldots,N\}$ satisfying either $i\ne0$
or $v_{\p}(F({\bf g}))\le0$, we define $L_{\p,i}:=L_{i}$; for the
remaining case, we define $L_{\p,i}:=L_{\tilde{F}}$. Hence we see
\begin{align}
L_{\p,i}(\Phi(\tilde{{\bf g}}))=F({\bf g})\qquad & \text{ if \ensuremath{i=0} and \ensuremath{v_{\p}(F({\bf g}))>0};}\label{L0}\\
L_{\p,i}(\Phi(\tilde{{\bf g}}))=\phi_{i}(\tilde{{\bf g}})\qquad & \text{otherwise.}\label{eq:Li}
\end{align}

By assumption,  those $\mathbf{g}^{\mathbf{i}}$ with $\mathbf{i}\in\mathbb{Z}_{\ge0}^{n}$
and $|\mathbf{i}|\le d$ are linearly nondegenerate over $V_{F}(r)$, thus
  $\Phi(\tilde{{\bf g}})\in\mathbb{P}^{N}(K)$
is linearly nondegenerate over $V_{F}(r)$.  Applying Theorem \ref{MSMT}
with $\mathbf{a}=\Phi(\tilde{{\bf g}})$ and $V_{\mathcal{L}}=V_{F}(1)$,
we have 
\begin{align}\label{useSMT3}
 & \sum_{\p\in S}\sum_{i=0}^{N}\left(v_{\p}(L_{\p,i}(\Phi(\tilde{{\bf g}}))-v_{\p}(\Phi(\tilde{{\bf g}}))-v_{\p}(L_{\p,i})\right)\cr
\le & \frac{d_{r}(N+1)}{d_{r-1}}\big(h(\Phi(\tilde{{\bf g}}))+(r+1)h(F)+\frac{Nd_{r}+d_{r}-1}{2}\max\{0,2\gen-2+|S|\}\big).
\end{align}
Together with \eqref{eq: sum_formula_Sunit}, \eqref{L0}, \eqref{eq:Li}
and \eqref{eq: vL<=00003D00003D00003D0}, we have the following estimate
for the left hand side of \eqref{useSMT3}

\begin{align*}
\sum_{i=0}^{N}\sum_{\p\in S}\left(v_{\p}(L_{\p,i}(\Phi(\tilde{{\bf g}}))-v_{\p}(\Phi(\tilde{{\bf g}}))-v_{\p}(L_{\p,i})\right)\ge\sum_{\p\in S}\big(v_{\p}^{0}(F({\bf g}))-(N+1)v_{\p}(\Phi(\tilde{{\bf g}})).
\end{align*}
Therefore, we can derive from \eqref{useSMT3} and \eqref{phiup-1}
that 
\begin{align*}
\sum_{\p\in S}v_{\p}^{0}(F({\bf g}))\le & (\frac{d_{r}}{d_{r-1}}-1)(N+1)dn\max_{1\le i\le n}h(g_{i})+\frac{d_{r}(N+1)}{d_{r-1}}(r+1)h(F)\\
 & +\frac{d_{r}(N+1)(Nd_{r}+d_{r}-1)}{2d_{r-1}}\max\{0,2\gen-2+|S|\}.
\end{align*}
\end{proof}

\subsection{Proof of Theorem \ref{movinggcdunit}}

\begin{proof}[Proof of Theorem \ref{movinggcdunit}] Let $\alpha$
and $\beta$ be one of the nonzero coefficients of $F$ and $G$ respectively,
and put ${\bf g}:=(g_{1},\hdots,g_{n})\in(\mathcal{O}_{S}^{*})^{n}$,
$1\le j\le n$. Since $v_{\p}^{0}(F({\bf g}))\le v_{\p}^{0}(\frac{1}{\alpha}F({\bf g}))+v_{\p}^{0}(\alpha)$
and $v_{\p}^{0}(G({\bf g}))\le v_{\p}^{0}(\frac{1}{\beta}G({\bf g}))+v_{\p}^{0}(\beta)$,
we have 
\begin{align*}
N_{S,{\rm gcd}}(F({\bf g}),G({\bf g}))\le N_{S,{\rm gcd}}(\frac{1}{\alpha}F({\bf g}),\frac{1}{\beta}G({\bf g}))+\tilde{h}(F)+\tilde{h}(G),
\end{align*}
and 
\begin{align*}
h_{{\rm gcd}}(F({\bf g}),G({\bf g}))\le h_{{\rm gcd}}(\frac{1}{\alpha}F({\bf g}),\frac{1}{\beta}G({\bf g}))+\tilde{h}(F)+\tilde{h}(G).
\end{align*}
Then by elementary reductions, which we omit, it suffices to prove
the theorem for $\frac{1}{\alpha}F$ and $\frac{1}{\beta}G$. Therefore,
we assume that, with respect to some fixed total ordering on the set
of monomials in %
\mbox{%
$K[x_{1},\ldots,x_{n}]$%
}, the coefficient attached to the largest monomial appearing in $F$
(resp. in %
\mbox{%
$G$%
}) is 1. In this case, %
\mbox{%
$\tilde{h}(F^{e})=h(F^{e})=eh(F)=e\tilde{h}(F)$%
} for every %
\mbox{%
$e\in\mathbb{N}$%
}, thus we may assume that %
\mbox{%
$F$%
} and %
\mbox{%
$G$%
} have the same degree %
\mbox{%
$d$%
} via replacing %
\mbox{%
$F$%
} (resp. %
\mbox{%
$G$%
}) by some of its powers.

Let $\epsilon>0$ be given. We first choose $m$ sufficiently large
so that $m\ge2d$ and 
\begin{align}
\frac{M'mn}{M}\le\frac{\epsilon}{4},\label{findm}
\end{align}
where %
\mbox{%
$M:=M_{m}:=2\binom{m+n-d}{n}-\binom{m+n-2d}{n}$%
} and $M':=M'_{m}:=\binom{m+n}{n}-M$; this is possible because $M_{m}=\frac{m^{n}}{n!}+O(m^{n-1})$
and %
\mbox{%
$M'=O(m^{n-2})$%
} (by the proof of Theorem \ref{Refinement}). By \eqref{infvr} we
may then choose a sufficiently large integer $r\in\mathbb{N}$ such
that 
\begin{align}
\frac{w}{u}-1\le\frac{\epsilon}{4mn},\label{wu}
\end{align}
where $w:=\dim_{{\bf k}}V_{F,G}(Mr)$ and $u:=\dim_{{\bf k}}V_{F,G}(Mr-M)$
(as in Theorem \ref{Refinement}); in the case where %
\mbox{%
$F(0,\ldots,0)\ne0$%
}, we further require that 
\begin{align}
\frac{w'}{u'}-1\le\frac{\epsilon}{8dn(N+1)},\label{wu'}
\end{align}
where $w':=\dim V_{F}(r)$, $u':=\dim V_{F}(r-1)$ and $N:=\binom{n+d}{n}-1$
(as in Theorem \ref{=00003D00003D00003D000024S=00003D00003D00003D000024 part}).

We first consider when those $\mathbf{g}^{\mathbf{i}}$ with $\mathbf{i}\in\mathbb{Z}_{\ge0}^{n}$
and $|\mathbf{i}|\le m$ are   linearly degenerate over $V_{F,G}(Mr+1)$, i.e. there is a non-trivial
relation 
\begin{align}
\sum_{\mathbf{i}}\alpha_{\mathbf{i}}\mathbf{g}^{\mathbf{i}}=0,\label{uniteq}
\end{align}
where the sum runs over those $\mathbf{i}\in\mathbb{Z}_{\ge0}^{n}$
and $|\mathbf{i}|\le m$, and $\alpha_{\mathbf{i}}\in V_{F,G}(Mr+1)$
for each $\mathbf{i}$ such that $\alpha_{\mathbf{i}_{0}}\ne0$ for
some $\mathbf{i}_{0}$. Then we have 
\begin{align}
\sum_{\mathbf{i}\ne\mathbf{i}_{0}}\frac{\alpha_{\mathbf{i}}}{\alpha_{\mathbf{i}_{0}}}\mathbf{g}^{\mathbf{i}-\mathbf{i}_{0}}=-1.\label{uniteq-1}
\end{align}
Since ${\bf g}\in(\mathcal{O}_{S}^{*})^{n}$ and the number of zeros
and poles of each $\alpha_{\mathbf{i}}$ appearing in \eqref{uniteq}
is bounded by $2h(\alpha_{\mathbf{i}})\le2(Mr+1)(\tilde{h}(F)+\tilde{h}(G))$,
we can apply Theorem \ref{BrMa} with some $S'$ including $S$ and
the zeros and poles of those $\alpha_{\mathbf{i}}$ appearing in \eqref{uniteq}
to get 
\[
h(\frac{\alpha_{\mathbf{i}}}{\alpha_{\mathbf{i}_{0}}}\mathbf{g}^{\mathbf{i}-\mathbf{i}_{0}})\le\tilde{c}\max\left\{ 0,2\gen-2+|S|+2(Mr+1)\binom{n+m}{n}(\tilde{h}(F)+\tilde{h}(G))\right\} ,
\]
where $\tilde{c}:=\frac{1}{2}\left(\binom{n+m}{n}-1\right)\left(\binom{n+m}{n}-2\right)$.
Then 
\begin{align}\label{multiht1}
h(\mathbf{g}^{\mathbf{i}-\mathbf{i}_{0}}) & \le h(\frac{\alpha_{\mathbf{i}}}{\alpha_{\mathbf{i}_{0}}})+h(\frac{\alpha_{\mathbf{i}}}{\alpha_{\mathbf{i}_{0}}}\mathbf{g}^{\mathbf{i}-\mathbf{i}_{0}})\cr
 & \le2(Mr+1)(\tilde{c}\binom{n+m}{n}+1)(\tilde{h}(F)+\tilde{h}(G))+\tilde{c}\max\{0,2\gen-2+|S|\},
\end{align}
which fits the assertion \eqref{multiheight11} with $(m_{1},\ldots,m_{n})=\mathbf{i}-\mathbf{i}_{0}$
and $\sum_{j=1}^{n}|m_{j}|\le2m$.

We now consider when  those $\mathbf{g}^{\mathbf{i}}$ with $\mathbf{i}\in\mathbb{Z}_{\ge0}^{n}$
and $|\mathbf{i}|\le m$ are linearly nondegenerate over $V_{F,G}(Mr+1)$. By Theorem \ref{Refinement}
combined with \eqref{findm} and \eqref{wu}, 
\begin{align*}
N_{S,{\rm gcd}}(F({\bf g}),G({\bf g})) & \le(\frac{M'mn}{M}+(\frac{w}{u}-1)mn)\max_{1\le i\le n}\{h(g_{i})\}+\tilde{c}_{1}(\tilde{h}(F)+\tilde{h}(G))+\tilde{c}_{2}\max\{0,2\gen-2+|S|\}\\
 & \le\frac{\epsilon}{2}\max_{1\le i\le n}\{h(g_{i})\}+\tilde{c}_{1}(\tilde{h}(F)+\tilde{h}(G))+\tilde{c}_{2}\max\{0,2\gen-2+|S|\},
\end{align*}
where $\tilde{c}_{1}$ and $\tilde{c}_{2}$ depend only on $(M_{m},r)$,
thus only on $\epsilon$. Hence, if 
\begin{align}
\max_{1\le i\le n}\{h(g_{i})\}\ge\frac{4}{\epsilon}\left(\tilde{c}_{1}(\tilde{h}(F)+\tilde{h}(G))+\tilde{c}_{2}\max\{0,2\gen-2+|S|\}\right),\label{heightbound1}
\end{align}
then 
\begin{align}
N_{S,{\rm gcd}}(F({\bf g}),G({\bf g}))\le\frac{3\epsilon}{4}\max_{1\le i\le n}\{h(g_{i})\}.\label{NSgcd}
\end{align}
We now estimate $h_{gcd}(F({\bf g}),G({\bf g}))$ using Theorem \ref{=00003D00003D00003D000024S=00003D00003D00003D000024 part}
with extra assumption that $F$ or $G$ does not vanish at the origin.
We may assume that $F(0,\hdots,0)\ne0$. Since $m\ge2d$ and  those $\mathbf{g}^{\mathbf{i}}$ with $\mathbf{i}\in\mathbb{Z}_{\ge0}^{n}$
and $|\mathbf{i}|\le m$ are linearly nondegenerate over $V_{F,G}(Mr+1)$, it is clear that
 those $\mathbf{g}^{\mathbf{i}}$ with $\mathbf{i}\in\mathbb{Z}_{\ge0}^{n}$
and $|\mathbf{i}|\le d$ are linearly nondegenerate over $V_{F}(r)$. Then by Theorem \ref{=00003D00003D00003D000024S=00003D00003D00003D000024 part}
and \eqref{wu'}, we have the the following 
\begin{align}
\sum_{\p\in S}v_{\p}^{0}(F({\bf g})) & \le\frac{\epsilon}{8}\max_{1\le i\le n}h(g_{i})+c_{1}'h(F)+c_{2}'\max\{0,2\gen-2+|S|\},\label{eq: S-part}
\end{align}
where $c_{1}':=\frac{w'(N+1)}{u'}(r+1)$ and $c_{2}':=\frac{w'(N+1)(Nw'+w'-1)}{2u'}$
with $N:=\binom{n+d}{n}-1$. Note that $c_{1}'$ and $c_{2}'$ depend
only on $(w',u')$, thus only on $\epsilon$. By \eqref{eq: S-part},
we see that if both \eqref{heightbound1} and 
\begin{align*}
\max_{1\le i\le n}\{h(g_{i})\}\ge\frac{8}{\epsilon}\left(c_{1}'h(F)+c_{2}'\max\{0,2\gen-2+|S|\}\right)
\end{align*}
hold, then 
\begin{align*}
\sum_{\p\in S}\min\{v_{\p}^{0}(F({\bf g})),v_{\p}^{0}(G({\bf g}))\}\le\sum_{\p\in S}v_{\p}^{0}(F({\bf g}))\le\frac{\epsilon}{4}\max_{1\le i\le n}\{h(g_{i})\},
\end{align*}
and hence together with \eqref{NSgcd}, we have 
\begin{align*}
h_{{\rm gcd}}(F({\bf g}),G({\bf g}))\le\epsilon\max_{1\le i\le n}\{h(g_{i})\}.
\end{align*}
\end{proof}

\begin{proof}[Proof of Theorem \ref{n=00003D00003D00003D00003D00003D00003D2gcdunit}]
Since $F,\,G\in{\bf k}[x_{1},x_{2}]$, we have that $V_{F,G}(r)=V_{F}(r)=\mathbf{k}$
for each $r\in\mathbb{N}$. Given $\epsilon>0$, we first choose $m$
sufficiently large satisfying \eqref{findm} with $n=2$. Suppose that  those $g_1^{i_1}g_2^{i_2}$ with $(i_1,i_2)\in\mathbb{Z}_{\ge0}^2$ and $i_1+i_2\le m$
 are linearly dependent over ${\bf k}$. Then there is a linear relation
\begin{align}
\sum_{\mathbf{j}=(j_{1},j_{2})}\alpha_{\mathbf{j}}g_{1}^{j_{1}}g_{2}^{j_{2}}=1,\label{uniteq0}
\end{align}
where $\alpha_{\mathbf{j}}\in{\bf k}^{*}$, $\mathbf{j}\in\mathbb{Z}^{2}\setminus\{(0,0)\}$
and $|j_{1}|+|j_{2}|\le2m$ for each appearing $\mathbf{j}=(j_{1},j_{2})$.
We may also assume that no proper subsum of the left hand side of
\eqref{uniteq0} vanishes. Consider the subgroup $J\subset\mathbb{Z}^{2}$
generated by those $\mathbf{j}$ appearing in \eqref{uniteq0}. If
$J$ has rank one, i.e., there exists $(m_{1},m_{2})\in\boldsymbol{\mathbb{Z}^{2}}\setminus\{(0,0)\}$
such that $(j_{1},j_{2})=\lambda_{\mathbf{j}}(m_{1},m_{2})$ with
$\lambda_{\mathbf{j}}\in\mathbb{Z}$ for every $\mathbf{j}=(j_{1},j_{2})$
appearing in \eqref{uniteq0}, then $|m_{1}|+|m_{2}|\le2m$ and 
\begin{align*}
\sum_{\mathbf{j}}\alpha_{\mathbf{j}}(g_{1}^{m_{1}}g_{2}^{m_{2}})^{\lambda_{\mathbf{j}}}=1,
\end{align*}
which implies that $g_{1}^{m_{1}}g_{2}^{m_{2}}\in\mathbf{k}$. For
the other cases, $J$ must have rank two, thus we can find $(j_{1},j_{2})$
and $(j_{1}',j_{2}')$ appearing in \eqref{uniteq0} such that $\mathbb{Q}\cdot(j_{1},j_{2})\ne\mathbb{Q}\cdot(j_{1}',j_{2}')$
(i.e. $(j_{1},j_{2})$ and $(j_{1}',j_{2}')$ are $\mathbb{Q}$-linearly
independent), and 
\begin{equation}
\max\{h(g_{1}^{j_{1}}g_{2}^{j_{2}}),h(g_{1}^{j_{1}'}g_{2}^{j_{2}'})\}\le\frac{1}{2}\left(\binom{m+2}{2}-1\right)\left(\binom{m+2}{2}-2\right)\max\{0,2\gen-2+|S|\}\label{eq: htb}
\end{equation}
by using Theorem \ref{BrMa}. Since $\mathbb{Q}\cdot(j_{1},j_{2})\ne\mathbb{Q}\cdot(j_{1}',j_{2}')$
and thus $j_{1}j_{2}'\ne j_{1}'j_{2}$, we note that $(j_{2}'j_{1}-j_{2}j_{1}',0)=j_{2}'(j_{1},j_{2})-j_{2}(j_{1}',j_{2}')$,
thus by \eqref{eq: htb} we have 
\begin{align*}
h(g_{1})\le h(g_{1}^{j_{2}'j_{1}-j_{2}j_{1}'})&\le|j_{2}'|h(g_{1}^{j_{1}}g_{2}^{j_{2}})+|j_{2}|h(g_{1}^{j_{1}'}g_{2}^{j_{2}'})\cr
&\le2m\left(\binom{m+2}{2}-1\right)\left(\binom{m+2}{2}-2\right)\max\{0,2\gen-2+|S|\}.
\end{align*}
With similar estimates for %
\mbox{%
$h(g_{2})$%
}, this implies that %
\mbox{%
$\max\{h(g_{1}),h(g_{2})\}\le c\max\{0,2\gen-2+|S|\}$%
}, where $c:=2m\left(\binom{m+2}{2}-1\right)\left(\binom{m+2}{2}-2\right)$
depends only on %
\mbox{%
$\epsilon$%
}.

In the case that %
those $g_1^{i_1}g_2^{i_2}$ with $(i_1,i_2)\in\mathbb{Z}_{\ge0}^2$ and $i_1+i_2\le m$ are  linearly independent over %
\mbox{%
${\bf k}$%
}, we shall conclude that %
\mbox{%
$N_{S,{\rm gcd}}(F(g_{1},g_{2}),G(g_{1},g_{2}))\le\epsilon\max\{h(g_{1}),h(g_{2})\}$%
}, and that %
\mbox{%
$h_{{\rm gcd}}(F(g_{1},g_{2}),G(g_{1},g_{2}))\le\epsilon\max\{h(g_{1}),h(g_{2})\}$%
} if we further assume that not both of %
\mbox{%
$F$%
} and %
\mbox{%
$G$%
} vanish at %
\mbox{%
$(0,0)$%
}. The corresponding part in the proof of Theorem %
\mbox{%
\ref{movinggcdunit}%
} works, but actually an easier proof suffices. We omit the details.
\end{proof}




\begin{proof}[Proof of Theorem \ref{movinggcdpower}] Let $S=S_{{\bf g}}:=\{\p\in C\,|\,v_{\p}(g_{i})\ne0\text{ for }1\le i\le n\}.$
Then, $g_{i}\in\mathcal{O}_{S}^{*}$ for each $1\le i\le n$, and
\begin{align}
|S|\le2\sum_{i=1}^{n}h(g_{i})\le2n\max_{1\le i\le n}\{h(g_{i})\}.\label{sizeS}
\end{align}
Let $\epsilon>0$. Suppose that our assertion (i) (resp. (ii)) does
not hold for some %
\mbox{%
$\ell$%
}. By Theorem %
\mbox{%
\ref{movinggcdunit}%
} applied to %
\mbox{%
$(g_{1}^{\ell},\hdots,g_{n}^{\ell})\in({\cal O}_{S}^{*})^{n}$%
}, there exist an integer %
\mbox{%
$m$%
}, positive constants %
\mbox{%
$c_{i}'$%
}, %
\mbox{%
$0\le i\le4$%
}, all depending only on %
\mbox{%
$\epsilon$%
}, such that we have either 
\begin{equation}
\max_{1\le i\le n}h(g_{i}^{\ell})\le c_{1}'(\tilde{h}(F)+\tilde{h}(G))+c_{2}'\max\{0,2\gen-2+|S|\},\label{eq:maxht}
\end{equation}
or 
\begin{align}
h(g_{1}^{{\ell}m_{1}}\cdots g_{n}^{{\ell}m_{n}})\le c_{3}'(\tilde{h}(F)+\tilde{h}(G))+c_{4}'\max\{0,2\gen-2+|S|\}\label{multiheight1}
\end{align}
for some integers %
\mbox{%
$m_{1},\hdots,m_{n}$%
}, not all zeros with %
\mbox{%
$\sum|m_{i}|\le2m$%
}. If %
\mbox{%
$g_{1}^{m_{1}}\cdots g_{n}^{m_{n}}\notin{\bf k}$%
}, then we must have %
\mbox{%
$\max_{1\le i\le n}h(g_{i})\ge1$%
} and %
\mbox{%
$h(g_{1}^{m_{1}}\cdots g_{n}^{m_{n}})\ge1$%
}. Hence %
\mbox{%
\eqref{sizeS}%
}, %
\mbox{%
\eqref{eq:maxht}%
} and %
\mbox{%
\eqref{multiheight1}%
} imply that \\
 \raisebox{-\belowdisplayshortskip}{%
\noindent\parbox[b]{1\linewidth}{%
\begin{align*}
\ell & \le(c_{1}'+c_{3}')(\tilde{h}(F)+\tilde{h}(G))+(c_{2}'+c_{4}')\max\{0,2\gen-2+|S|\}\\
 & \le(c_{1}'+c_{3}')(\tilde{h}(F)+\tilde{h}(G))+2(c_{2}'+c_{4}')(\gen+n\max_{1\le i\le n}\{h(g_{i})\}).
\end{align*}
}}\\
 This shows that our desired conclusion holds with %
\mbox{%
$c_{1}:=c_{1}'+c_{3}'$%
} and %
\mbox{%
$c_{2}:=2(c_{2}'+c_{4}')$%
}. \end{proof}


\end{document}